\documentclass[11pt]{article}
\usepackage[utf8]{inputenc}
\usepackage[T1]{fontenc}
\usepackage{lmodern}
\usepackage[british]{babel}
\usepackage{amsmath,amssymb,amsthm,mathtools}
\usepackage[a4paper,margin=2.5cm]{geometry}
\usepackage{enumitem}
\usepackage{graphicx}
\usepackage[colorlinks=true,linkcolor=blue,citecolor=blue,urlcolor=blue]{hyperref}

\newtheorem{theorem}{Theorem}[section]
\newtheorem{proposition}[theorem]{Proposition}
\newtheorem{lemma}[theorem]{Lemma}

\theoremstyle{definition}
\newtheorem{assumption}[theorem]{Assumption}
\newtheorem{definition}[theorem]{Definition}
\theoremstyle{remark}
\newtheorem{remark}[theorem]{Remark}
\newtheorem{example}[theorem]{Example}
\numberwithin{equation}{section}

\newcommand{\R}{\mathbb{R}}

\newcommand{\E}{\mathbb{E}}
\newcommand{\PP}{\mathbb{P}}
\newcommand{\eps}{\varepsilon}
\newcommand{\WW}{\mathbb{W}}
\newcommand{\F}{\mathcal{F}}
\newcommand{\Cr}{\mathcal{C}_\rho}
\newcommand{\U}{\mathcal{U}}
\newcommand{\Var}{\operatorname{Var}}
\newcommand{\dd}{\mathrm{d}}

\title{Parabolic equations with large long-range random potentials of arbitrary Hermite rank: convergence to Hermite-driven Feynman--Kac limits}
\author{Atef Lechiheb\thanks{Toulouse School of Economics, Universit\'e Toulouse Capitole, Toulouse, France. E-mail: \texttt{atef.lechiheb@tse-fr.eu}}}
\date{}

\begin{document}
\maketitle

\begin{abstract}
We consider the one-dimensional heat equation with a large, rapidly oscillating random potential, obtained by applying a polynomial of arbitrary Hermite rank to a stationary Gaussian process with long-range dependence. When the correlations decay slowly enough, the solution does not homogenise to a deterministic limit. We prove that it converges in law, as a random field, to a Feynman--Kac functional driven by a Hermite process, which is fractional Brownian motion for Hermite rank one and the Rosenblatt process for Hermite rank two. The proof rests on an identity which expresses the exponent of the Feynman--Kac formula as a continuous functional of the integrated potential, so that neither local times of fractional processes nor exponential moments with respect to the potential are needed. We also show that the limit has moments of every order for Hermite rank one, and for Hermite rank two with a negative second Hermite coefficient, whereas for a positive coefficient its mean is infinite, although it is finite almost surely. A numerical experiment illustrates the convergence of the integrated potential to the Rosenblatt law.
\end{abstract}

\noindent\textbf{Keywords:} stochastic homogenisation; long-range dependence; Hermite process; Rosenblatt process; Feynman--Kac formula.

\smallskip
\noindent\textbf{MSC 2020:} 60H15; 60F05; 60G22; 35B27; 35R60.

\section{Introduction}
\label{sec:intro}

Consider the heat equation on the real line with a random potential,
\begin{equation}
\label{eq:pde}
\partial_tu_\eps(t,x)=\frac12\,\partial_{xx}u_\eps(t,x)+\eps^{-\gamma}\,V\Bigl(\frac x\eps\Bigr)\,u_\eps(t,x),
\qquad u_\eps(0,x)=\varphi(x),
\end{equation}
where $V$ is a stationary, centred random field, $\eps>0$ is the correlation length and $\gamma>0$ is chosen so that the potential has an effect of order one as $\eps\to0$. Since $V$ is centred, a naive averaging argument would suggest that the potential disappears in the limit; the large amplitude $\eps^{-\gamma}$ prevents this, and the limit depends on the correlation structure of $V$. When the correlations of $V$ are integrable, the relevant amplitude is $\gamma=1/2$ and, in dimension one, the solution converges to the solution of a stochastic heat equation with multiplicative white noise; see Pardoux and Piatnitski~\cite{PardouxPiatnitski2012} and Iftimie, Pardoux and Piatnitski~\cite{IftimiePardouxPiatnitski2008}. Large potentials of this type, in all dimensions, are studied by Bal~\cite{Bal2009,Bal2010} and by Gu and Bal~\cite{GuBal2016}, who also survey the field in~\cite{BalGu2015}.

This paper is concerned with potentials with long-range dependence. We take $V=\Phi(g)$, where $g$ is a stationary Gaussian process whose covariance decays like $|x|^{-\alpha}$ with $\alpha\in(0,1)$, and $\Phi$ is a polynomial of Hermite rank $m\ge1$. The correlations of $V$ then decay like $|x|^{-m\alpha}$, and they are not integrable when $m\alpha<1$, which is the regime we consider. In this regime the integrated potential $\int_0^yV(z)\,\dd z$ satisfies a non-central limit theorem: after normalisation, it converges to a Hermite process of order $m$ and Hurst index $H=1-m\alpha/2\in(\frac12,1)$, which is fractional Brownian motion when $m=1$ and the Rosenblatt process when $m=2$; this is the theorem of Taqqu~\cite{Taqqu1975,Taqqu1979} and of Dobrushin and Major~\cite{DobrushinMajor1979}. The corresponding amplitude in~\eqref{eq:pde} is $\gamma=m\alpha/2$.

For $m=1$, the situation is well understood. For a real Gaussian potential with integrable correlations, Bal~\cite{Bal2009} proved, by a Duhamel expansion, the convergence to a stochastic parabolic equation with multiplicative white noise in the Stratonovich sense, when the order of the elliptic operator exceeds the dimension. For a real Gaussian potential, that is $\Phi(z)=z$, with correlations decaying like $|x|^{-\alpha}$, the convergence in distribution to a stochastic heat equation with a multiplicative noise of covariance proportional to $|x-y|^{-\alpha}$ is proved by Bal~\cite[Theorem~2.4 and Remark~2.6]{Bal2011}, again by a Duhamel expansion; see also the review~\cite[Theorem~3.3]{BalGu2015}. Komorowski and Nieznaj~\cite[Theorems~1 and~2]{KomorowskiNieznaj2010} consider, in every dimension $d\ge1$, a real Gaussian potential whose spectral density is a power of $|p|$ times a compactly supported cut-off, with correlations decaying like $|x|^{2\alpha'-2}$ for some $\alpha'\in(0,1)$; for an initial condition in the Schwartz class and $\gamma=1-\alpha'$, they prove, by a Duhamel expansion, the convergence of all joint moments of the Fourier transform of the solution and the convergence in law of the solution in $C([0,\infty);\mathcal S'(\R^d))$ to the Duhamel solution of a heat equation with a Gaussian, distribution-valued potential. In our notation, $\alpha=2-2\alpha'$ and $\gamma=\alpha/2$. For a potential $\Phi(g)$ of Hermite rank one, Gu and Bal~\cite[Theorem~2.11 and Remark~2.12]{GuBal2016} treat the equation with the imaginary potential $i\eps^{-\alpha/2}V(x/\eps)$; the imaginary unit makes the Feynman--Kac exponential bounded and avoids all integrability issues. The limit is then the Feynman--Kac representation of a parabolic Anderson model with a time-independent fractional noise in space; Feynman--Kac representations of parabolic Anderson models with fractional noises are studied in~\cite{HuNualartSong2011,HuHuangNualartTindel2015}. For one-dimensional elliptic equations with long-range coefficients, Gu and Bal~\cite{GuBal2012} construct a class of random coefficients for which the rescaled corrector converges in distribution to a stochastic integral with respect to the Rosenblatt process; the extension to integrals with respect to Hermite processes of arbitrary order is due to Lechiheb, Nourdin, Zheng and Haouala~\cite{LechihebNourdinZhengHaouala2018}; the fluctuations of pseudo-elliptic equations with long-range potentials are studied in~\cite{LechihebHaouala2018}. Closest to the present work is the paper of Lang and Xanh~\cite{LangXanh1983} on random walks in strongly correlated random sceneries. They consider a random walk $(S_j)$ on $\mathbb Z^\nu$ in the domain of attraction of a stable law of index $\beta_0\in(0,2]$, independent of a scenery $\xi(x)=H_k(X_x)$, where $X$ is a stationary Gaussian field whose correlations decay like $|x|^{-\alpha}$, and they prove, under the conditions $\alpha k<\nu$ and $\alpha k<\beta_0$, that the normalised sums $\sum_{j\le nt}\xi(S_j)$ converge in law to a self-similar process, written as a multiple Wiener--Itô integral of the Fourier transform of the occupation measure of the limiting stable process. For $\nu=1$, $\beta_0=2$ and $k=m$, their theorem is, after the Brownian rescaling $B_s=\eps B'_{s/\eps^2}$, the lattice analogue of the convergence in law of the exponent $\int_0^t V_\eps(B_s)\,\dd s$ of the Feynman--Kac formula~\eqref{eq:FK} for $\Phi=H_m$, and their limit is, in spectral form, the exponent of our limit; they interpret it formally as the integral of the local time against the limiting field~\cite[Section~1, comment~(ii)]{LangXanh1983}, an interpretation which Proposition~\ref{prop:young} makes rigorous as a Young integral. Their result concerns the law of the additive functional, jointly in the walk and in the scenery. The convergence of the solution of~\eqref{eq:pde} requires in addition the convergence of the average over the walk for a fixed scenery, at all points $(t,x)$ simultaneously, and the passage to the limit inside an exponential, which does not follow from a convergence in law of the exponent. Dombry and Guillotin-Plantard~\cite{DombryGuillotin2009} developed a functional approach to random walks in random sceneries, in which the additive functional is written as the integral of the local time against the scenery. To the best of our knowledge, the convergence of the solution of the parabolic equation~\eqref{eq:pde} with a real potential of Hermite rank $m\ge2$ and long-range dependence has not been established; for $m=1$, our results extend the Gaussian case to polynomial potentials of rank one.

\paragraph{Results.} We prove three results. First, the integrated potential, normalised by $\eps^{-m\alpha/2}$, converges in law to a multiple of the Hermite process, in a space of continuous functions with a polynomial weight at infinity. Second, the solution $u_\eps$, defined by the Feynman--Kac formula, converges in law, as a random field on $[0,T]\times\R$ with the topology of locally uniform convergence, to an explicit random field $u$; on a natural coupling, the convergence holds in probability. The limit has the form
\[
u(t,x)=\E^B\Bigl[\varphi(x+B_t)\exp\Bigl(\beta\int_\R L^x_t(y)\,\dd Z(y)\Bigr)\Bigr],
\]
where $B$ is a Brownian motion independent of everything else, $L^x_t$ is the local time of $x+B$, $Z$ is the Hermite process and the integral is a Young integral. Third, for $m=2$ and a positive second Hermite coefficient of $\Phi$, we show that $\E[u(t,x)]=+\infty$ for every $t>0$, whenever $\varphi(x)>0$, although $u(t,x)$ is finite almost surely. The mechanism is the following: conditionally on $B$, the exponent belongs to the second Wiener chaos, and when the Brownian path stays in a small interval, the largest eigenvalue of the corresponding operator exceeds $\frac12$, so that its exponential has no mean. In contrast, $u(t,x)$ has finite moments of every order when $m=1$, and when $m=2$ and $c_2<0$: conditionally on $B$, the exponent is then Gaussian, or a second-chaos variable with non-positive eigenvalues, and its conditional exponential moments are controlled by those of the double integral $\int_0^t\int_0^t|B_s-B_r|^{2H-2}\,\dd s\,\dd r$, which are finite because the maximum of the Brownian local time has exponential moments of every order. In particular, when $m=2$ and $c_2>0$, the limit cannot be described through its moments.

\paragraph{Method.} The main difficulty is the passage to the limit inside the exponential of the Feynman--Kac formula. We avoid local times of fractional processes and exponential moments with respect to the potential altogether. If $W_\eps$ denotes the integrated potential and $F_\eps$ its primitive, Itô's formula for the Brownian motion $B$ gives
\[
\int_0^t\eps^{-m\alpha/2}\,V\Bigl(\frac{x+B_s}\eps\Bigr)\dd s
=2\Bigl(F_\eps(x+B_t)-F_\eps(x)-\int_0^tW_\eps(x+B_s)\,\dd B_s\Bigr).
\]
The right-hand side makes sense for every continuous function $W_\eps$ with sub-linear growth, and it depends continuously on $W_\eps$ for a weighted uniform norm. Exponential moments are only needed with respect to the Brownian motion $B$, for a fixed path of the integrated potential, and they follow from its sub-linear growth. The convergence of $u_\eps$ then follows from the convergence of the integrated potential by the continuous mapping theorem. The Young integral form of the limit is recovered afterwards by an approximation argument.

\paragraph{Organisation.} Section~\ref{sec:setting} describes the setting and the assumptions, and Section~\ref{sec:results} states the main results. Section~\ref{sec:prelim} collects the tools we use. The convergence of the integrated potential is proved in Section~\ref{sec:proof_W}, the convergence of the solution in Section~\ref{sec:proof_u}, the Young integral representation in Section~\ref{sec:young} and the results on the moments of the limit in Section~\ref{sec:mean}. Section~\ref{sec:numerics} presents a numerical illustration and Section~\ref{sec:conclusion} discusses extensions.

\section{Setting and assumptions}
\label{sec:setting}

\subsection{The Gaussian process}

All random objects are defined on a complete probability space $(\Omega,\F,\PP)$. Fourier transforms are normalised as $\hat f(\eta)=\int_\R e^{-is\eta}f(s)\,\dd s$, so that $\|f\|_{L^2(\R^q)}^2=(2\pi)^{-q}\|\hat f\|_{L^2(\R^q)}^2$ for $f\in L^2(\R^q)$.

\begin{assumption}[Gaussian process]
\label{ass:g}
The process $g=(g(z))_{z\in\R}$ is a centred, stationary Gaussian process with covariance $R(x)=\E[g(0)g(x)]$, $R(0)=1$, satisfying the following conditions, for some $\alpha\in(0,1)$.
\begin{enumerate}[label=\textup{(G\arabic*)}]
\item\label{G1} $g$ has a spectral density of the form $f(\eta)=|\eta|^{\alpha-1}\,b(\eta)$, that is, $R(x)=\int_\R e^{ix\eta}f(\eta)\,\dd\eta$, where $b$ is even, non-negative, bounded and continuous at $0$, with $b(0)>0$.
\item\label{G2} There are constants $C_R\ge1$ and $\kappa>0$ such that
\[
|R(x)|\le C_R\,(1+|x|)^{-\alpha}\quad\text{for all }x\in\R,
\qquad
\lim_{|x|\to\infty}|x|^{\alpha}R(x)=\kappa .
\]
\item\label{G3} There is $\delta_0\in(0,1]$ such that $\int_\R|\eta|^{2\delta_0}f(\eta)\,\dd\eta<\infty$.
\end{enumerate}
\end{assumption}

Conditions~\ref{G1} and~\ref{G2} are related. Comparing the limit in Lemma~\ref{lem:second_moments}(ii) with a direct computation of $m!\,\|K_0^1\|^2_{L^2(\R^m)}$, based on the identity
\[
\int_\R e^{-ir\eta}|\eta|^{\alpha-1}\,\dd\eta=2\Gamma(\alpha)\cos\Bigl(\frac{\pi\alpha}2\Bigr)|r|^{-\alpha},\qquad r\ne0,
\]
in the sense of improper integrals, one finds that necessarily $\kappa=2\Gamma(\alpha)\cos(\pi\alpha/2)\,b(0)$. We keep both conditions because~\ref{G1} alone does not imply the limit in~\ref{G2} without further regularity of $b$. Condition~\ref{G3} only ensures that $g$ has a continuous version. Indeed, since $1-\cos r\le\min(2,r^2/2)\le2|r|^{2\delta_0}$,
\[
\E|g(z)-g(z')|^2=2\int_\R\bigl(1-\cos((z-z')\eta)\bigr)f(\eta)\,\dd\eta\le4|z-z'|^{2\delta_0}\int_\R|\eta|^{2\delta_0}f(\eta)\,\dd\eta,
\]
and, $g$ being Gaussian, Kolmogorov's criterion applies. This holds for every centred Gaussian process with covariance $R$.

\begin{example}
\label{ex:cauchy}
The generalised Cauchy covariance $R(x)=(1+x^2)^{-\alpha/2}$ satisfies Assumption~\ref{ass:g} with $\kappa=1$ and $C_R=2^{\alpha/2}$, since $(1+|x|)^2\le2(1+x^2)$. Its spectral density is $f(\eta)=c_\alpha|\eta|^{(\alpha-1)/2}K_{(1-\alpha)/2}(|\eta|)$, where $K_\nu$ is the modified Bessel function of the second kind and $c_\alpha>0$; since $K_\nu(r)\sim\frac12\Gamma(\nu)(r/2)^{-\nu}$ as $r\downarrow0$ and $K_\nu$ decays exponentially at infinity, $f$ has the form~\ref{G1} with $b$ bounded and continuous. Condition~\ref{G3} holds for every $\delta_0\in(0,1]$, because $f$ decays exponentially.
\end{example}

\paragraph{White noise representation.} Let $\WW=(\WW(s))_{s\in\R}$ be a two-sided Brownian motion, and let $I_q$ denote the multiple Wiener--Itô integral of order $q\ge1$ with respect to $\WW$, defined on symmetric functions of $L^2(\R^q)$ and extended to $L^2(\R^q)$ by symmetrisation~\cite[Section~1.1.2]{Nualart2006}. The function $\eta\mapsto\sqrt{2\pi f(\eta)}$ is even, real and square integrable, since $\int f=R(0)=1$. Let $\mathsf a\in L^2(\R)$ be the real and even function with $\hat{\mathsf a}=\sqrt{2\pi f}$. Then $\int\mathsf a^2=(2\pi)^{-1}\int2\pi f=1$, and the process $\int_\R\mathsf a(z-s)\,\dd\WW(s)$, $z\in\R$, is centred Gaussian with covariance
\[
\int_\R\mathsf a(z-s)\,\mathsf a(z'-s)\,\dd s
=\frac1{2\pi}\int_\R e^{i(z-z')\eta}\,|\hat{\mathsf a}(\eta)|^2\,\dd\eta
=R(z-z').
\]
Since only the law of $g$ matters below, we assume from now on that $g$ is the continuous version of this process, which exists by~\ref{G3}. For $\eps>0$, the process $\WW^\eps(s):=\eps^{-1/2}\WW(\eps s)$ is again a two-sided Brownian motion, and
\begin{equation}
\label{eq:g_eps}
g_\eps(z):=\int_\R\mathsf a(z-s)\,\dd\WW^\eps(s),\qquad z\in\R,
\end{equation}
has the same law as $g$; we use its continuous version. The family $(g_\eps)_{\eps>0}$ lives on a single probability space, which allows us to prove convergence in $L^2(\Omega)$ rather than only in law.

\subsection{The potential}

Let $H_q$ be the Hermite polynomials, $H_0=1$, $H_1(z)=z$, $H_2(z)=z^2-1$, $H_{q+1}(z)=zH_q(z)-qH_{q-1}(z)$. If $N$ is a standard Gaussian variable, $\E[H_q(N)H_p(N)]=q!\,\mathbf 1_{\{p=q\}}$.

\begin{assumption}[Nonlinearity]
\label{ass:Phi}
$\Phi$ is a polynomial of Hermite rank $m\ge1$ and degree $M\ge m$, that is,
\[
\Phi(z)=\sum_{q=m}^{M}\frac{c_q}{q!}\,H_q(z),\qquad c_q=\E\bigl[\Phi(N)H_q(N)\bigr],\qquad c_m\ne0,
\]
and $m\alpha<1$, where $\alpha$ is the exponent of Assumption~\ref{ass:g}.
\end{assumption}

In particular $\E[\Phi(N)]=c_0=0$, so that the potential $V:=\Phi(g)$ is centred. We set
\begin{equation}
\label{eq:H_beta}
H:=1-\frac{m\alpha}2\in\Bigl(\frac12,1\Bigr),
\qquad
\beta:=c_m\sqrt{\frac{\kappa^m}{m!\,H(2H-1)}} .
\end{equation}
For instance, $\Phi(z)=z$ has rank one, with $\beta=\sqrt{\kappa/(H(2H-1))}$, and $\Phi=H_2$ has rank two, with $c_2=2$ and $\beta=2\sqrt{\kappa^2/(2H(2H-1))}$.

\subsection{The equation and its Feynman--Kac solution}

Let $\varphi:\R\to\R$ be bounded and continuous. The rescaled potential is $V_\eps(y):=\eps^{-m\alpha/2}\Phi(g_\eps(y/\eps))$, which has the same law as $\eps^{-m\alpha/2}\Phi(g(y/\eps))$ as a process. Let $B=(B_s)_{s\ge0}$ be a standard Brownian motion independent of $\WW$, defined on a probability space $(\Omega^B,\F^B,\PP^B)$, with expectation $\E^B$. We work on the product space and write $\E^B$ for the integral with respect to $B$ only. Since $g_\eps$ has the same law as $g$, the laws of all the objects below are those of the original problem, and only these laws enter our results. As in~\cite{GuBal2016,PardouxPiatnitski2012}, we define the solution of~\eqref{eq:pde} with $V(\cdot/\eps)$ replaced by $\Phi(g_\eps(\cdot/\eps))$ and $\gamma=m\alpha/2$ by the Feynman--Kac formula
\begin{equation}
\label{eq:FK}
u_\eps(t,x):=\E^B\Bigl[\varphi(x+B_t)\exp\Bigl(\int_0^tV_\eps(x+B_s)\,\dd s\Bigr)\Bigr],
\qquad t\ge0,\ x\in\R .
\end{equation}
We show in Lemma~\ref{lem:U_well_defined} that the expectation is finite almost surely. We do not discuss in which sense~\eqref{eq:FK} solves~\eqref{eq:pde}: as in~\cite{GuBal2016,PardouxPiatnitski2012}, the Feynman--Kac formula is taken as the definition of the solution.

The object that governs the limit is the integrated potential
\begin{equation}
\label{eq:W_eps}
W_\eps(y):=\int_0^yV_\eps(z)\,\dd z=\eps^{-m\alpha/2}\int_0^y\Phi\Bigl(g_\eps\Bigl(\frac z\eps\Bigr)\Bigr)\dd z,
\qquad y\in\R,
\end{equation}
with the convention $\int_0^y=-\int_y^0$ for $y<0$. It is a random $C^1$ function with $W_\eps(0)=0$.

\subsection{The Hermite process}

For $y\in\R$ and $\zeta\in\R$, let $e_y(\zeta):=\int_0^ye^{-iu\zeta}\,\dd u$, and let $\psi(\eta):=|\eta|^{(\alpha-1)/2}$. We show in Lemma~\ref{lem:kernel_L2} that the function
\begin{equation}
\label{eq:K0}
\widehat{K_0^y}(\eta_1,\dots,\eta_m):=(2\pi b(0))^{m/2}\,e_y(\eta_1+\dots+\eta_m)\prod_{j=1}^m\psi(\eta_j)
\end{equation}
belongs to $L^2(\R^m)$ when $m\alpha<1$. It satisfies $\overline{\widehat{K_0^y}(\eta)}=\widehat{K_0^y}(-\eta)$, so it is the Fourier transform of a real function $K_0^y\in L^2(\R^m)$, which is symmetric in its $m$ variables.

\begin{definition}
\label{def:hermite}
Let $\sigma_m:=\sqrt{m!\,\kappa^m/(H(2H-1))}$. The process
\[
Z(y):=\sigma_m^{-1}\,I_m\bigl(K_0^y\bigr),\qquad y\in\R,
\]
taken in its continuous version (Proposition~\ref{prop:Z}), is called the Hermite process of order $m$ and Hurst index $H$.
\end{definition}

Proposition~\ref{prop:Z} below shows that $Z$ has a continuous version, belongs to the $m$-th Wiener chaos of $\WW$, is self-similar with index $H$, has stationary increments, and has the covariance of fractional Brownian motion, $\E[Z(y)Z(y')]=\frac12(|y|^{2H}+|y'|^{2H}-|y-y'|^{2H})$. This spectral representation of the Hermite process goes back to Dobrushin and Major~\cite{DobrushinMajor1979} and Taqqu~\cite{Taqqu1979}; its equality in finite-dimensional distributions with the time-domain representation is proved in~\cite[Theorem~6.3]{Taqqu1979}, and the time-domain, spectral and finite-interval representations are collected in~\cite[Theorem~1.1]{PipirasTaqqu2010}; see also~\cite[Chapter~4]{PipirasTaqqu2017}. For $m=1$, $Z$ is fractional Brownian motion with Hurst index $H$, and for $m=2$ it is the Rosenblatt process~\cite{Taqqu1975,Tudor2008}. Apart from the properties listed in Proposition~\ref{prop:Z}, our proofs use the kernel $K_0^y$ only in Section~\ref{sec:mean}.

\section{Main results}
\label{sec:results}

\subsection{The functional space}

For $\rho\in(0,1)$, let $\Cr$ be the space of continuous functions $w:\R\to\R$ such that $w(0)=0$ and $|w(y)|/(1+|y|)^\rho\to0$ as $|y|\to\infty$, with the norm
\[
\|w\|_\rho:=\sup_{y\in\R}\frac{|w(y)|}{(1+|y|)^\rho}.
\]
The map $w\mapsto w/(1+|\cdot|)^\rho$ is an isometry from $\Cr$ onto a closed subspace of the space of continuous functions vanishing at infinity, so $\Cr$ is a separable Banach space. Convergence in $\Cr$ implies uniform convergence on compact sets.

\subsection{The integrated potential}

\begin{theorem}[Convergence of the integrated potential]
\label{thm:W}
Under Assumptions~\ref{ass:g} and~\ref{ass:Phi}, let $H$ and $\beta$ be given by~\eqref{eq:H_beta} and let $Z$ be the Hermite process of Definition~\ref{def:hermite}. Then:
\begin{enumerate}[label=\textup{(\roman*)}]
\item for every $y\in\R$, $W_\eps(y)\to\beta Z(y)$ in $L^2(\Omega)$ as $\eps\to0$;
\item for every $\rho\in(H,1)$, $W_\eps$ belongs to $\Cr$ almost surely, and $W_\eps\to\beta Z$ in law in $\Cr$ as $\eps\to0$.
\end{enumerate}
\end{theorem}

Part~(i) is a form of the non-central limit theorem of Taqqu~\cite{Taqqu1979} and Dobrushin and Major~\cite{DobrushinMajor1979}, for which we give a short proof adapted to our assumptions. Part~(ii) adds tightness in the weighted space $\Cr$, which controls the growth of $W_\eps$ at infinity uniformly in $\eps$; this is what the Feynman--Kac functional requires.

\subsection{The Feynman--Kac functional and the limit of the solution}

For $w\in\Cr$, let $F_w(y):=\int_0^yw(z)\,\dd z$ and, for $t\ge0$ and $x\in\R$,
\begin{equation}
\label{eq:Lambda}
\Lambda_t^x(w):=2\Bigl(F_w(x+B_t)-F_w(x)-\int_0^tw(x+B_s)\,\dd B_s\Bigr),
\end{equation}
where the stochastic integral is the Itô integral with respect to $B$. We define
\begin{equation}
\label{eq:U}
\U_{t,x}(w):=\E^B\Bigl[\varphi(x+B_t)\,e^{\Lambda_t^x(w)}\Bigr].
\end{equation}
Lemma~\ref{lem:exp_bound} below shows that $\U_{t,x}(w)$ is finite for every $w\in\Cr$, $\rho<1$. The link with~\eqref{eq:FK} is the following identity, which is Itô's formula for $F_{W_\eps}\in C^2(\R)$: almost surely,
\begin{equation}
\label{eq:Ito_identity}
\int_0^tV_\eps(x+B_s)\,\dd s=\Lambda_t^x(W_\eps),
\qquad\text{so that}\qquad
u_\eps(t,x)=\U_{t,x}(W_\eps).
\end{equation}
Indeed, $F_{W_\eps}'=W_\eps$ and $F_{W_\eps}''=V_\eps$ is continuous, so that
\[
F_{W_\eps}(x+B_t)=F_{W_\eps}(x)+\int_0^tW_\eps(x+B_s)\,\dd B_s+\frac12\int_0^tV_\eps(x+B_s)\,\dd s .
\]

\begin{theorem}[Convergence of the solution]
\label{thm:u}
Under Assumptions~\ref{ass:g} and~\ref{ass:Phi}, let $\varphi$ be bounded and continuous, let $T>0$, and define
\[
u(t,x):=\U_{t,x}(\beta Z),\qquad t\in[0,T],\ x\in\R .
\]
Then $u$ and $u_\eps$, $\eps>0$, have continuous versions on $[0,T]\times\R$, and $u_\eps\to u$ in law in $C([0,T]\times\R)$, endowed with the topology of uniform convergence on compact sets. In particular, for all $(t_1,x_1),\dots,(t_k,x_k)$,
\[
\bigl(u_\eps(t_1,x_1),\dots,u_\eps(t_k,x_k)\bigr)\longrightarrow\bigl(u(t_1,x_1),\dots,u(t_k,x_k)\bigr)\quad\text{in law.}
\]
Moreover, on the probability space of Section~\ref{sec:setting}, where $W_\eps$ and $Z$ are defined from the same Brownian motion $\WW$, $u_\eps\to u$ in probability in $C([0,T]\times\R)$. If $\varphi\ge0$ and $\varphi\not\equiv0$, then $u(t,x)\in(0,\infty)$ almost surely for every $t\in(0,T]$ and $x\in\R$.
\end{theorem}

The limit can be written in the more familiar form of a Feynman--Kac formula with a Young integral. Let $L^x_t(y)$ denote the local time at level $y$ and time $t$ of the Brownian motion $x+B$, in its version which is jointly continuous in $(t,y)$; for fixed $t$, $y\mapsto L^x_t(y)$ is Hölder continuous of every order $\gamma<\frac12$ and has compact support~\cite[Chapter~VI, Section~1]{RevuzYor1999}.

\begin{proposition}[Young integral representation]
\label{prop:young}
Let $\rho\in(0,1)$, and let $w\in\Cr$ be locally Hölder continuous of some order $\theta>\frac12$. Then, for all $t\ge0$ and $x\in\R$, almost surely,
\[
\Lambda_t^x(w)=\int_\R L_t^x(y)\,\dd w(y),
\]
where the right-hand side is a Young integral on any compact interval containing the support of $L^x_t$. In particular, almost surely,
\[
u(t,x)=\E^B\Bigl[\varphi(x+B_t)\exp\Bigl(\beta\int_\R L_t^x(y)\,\dd Z(y)\Bigr)\Bigr].
\]
\end{proposition}

The Young integral is well defined because $L^x_t$ is Hölder of every order $\gamma<\frac12$ and $Z$ of every order $\theta<H$, and one can choose $\gamma<\frac12$ and $\theta<H$ with $\gamma+\theta>1$, since $H>\frac12$; for instance $\gamma=\frac12-\frac{H-1/2}4$ and $\theta=H-\frac{H-1/2}4$.

\subsection{Moments of the limit}

\begin{theorem}[Infinite mean for $m=2$ and $c_2>0$]
\label{thm:mean}
Under Assumptions~\ref{ass:g} and~\ref{ass:Phi}, let $m=2$ and $c_2>0$. Let $\varphi\ge0$ be bounded and continuous, $t>0$ and $x\in\R$ with $\varphi(x)>0$. Then $u(t,x)<\infty$ almost surely, and $\E[u(t,x)]=+\infty$.
\end{theorem}

\begin{proposition}[Finite moments for $m=1$, and for $m=2$ and $c_2<0$]
\label{prop:moments}
Under Assumptions~\ref{ass:g} and~\ref{ass:Phi}, assume that either $m=1$, or $m=2$ and $c_2<0$. Let $\varphi$ be bounded and continuous, $t\ge0$ and $x\in\R$. Then, for every $p\ge1$,
\begin{equation}
\label{eq:moment_bound}
\E\bigl[|u(t,x)|^p\bigr]\le\|\varphi\|_\infty^p\,\E^B\Bigl[\exp\Bigl(\frac{p^2\beta^2H(2H-1)}2\int_0^t\int_0^t|B_s-B_r|^{2H-2}\,\dd s\,\dd r\Bigr)\Bigr]<\infty .
\end{equation}
If $m=1$ and $\varphi\ge0$, then~\eqref{eq:moment_bound} holds with equality for $p=1$ and $\varphi(x+B_t)$ inside the expectation in place of $\|\varphi\|_\infty$.
\end{proposition}

\begin{remark}
\label{rem:mean}
\begin{enumerate}[label=(\alph*)]
\item Theorem~\ref{thm:mean} and Proposition~\ref{prop:moments} give a dichotomy for $m\le2$: the limit $u(t,x)$ has moments of every order, except when $m=2$ and $c_2>0$, in which case its mean is infinite. The mechanism is the following. Conditionally on $B$, the exponent $\beta\int L^x_t\,\dd Z$ is a centred Gaussian variable when $m=1$, and a second-chaos variable $\sum_k\lambda_k(\xi_k^2-1)$ when $m=2$. Its conditional exponential moment is finite when all the $\lambda_k$ are smaller than $\frac12$, which is the case when $c_2<0$, since then all the $\lambda_k$ are non-positive. When $c_2>0$ and the Brownian path stays in a small interval, the largest $\lambda_k$ exceeds $\frac12$.
\item In both cases of Proposition~\ref{prop:moments}, the conditional variance of the exponent is $\beta^2H(2H-1)\int_0^t\int_0^t|B_s-B_r|^{2H-2}\,\dd s\,\dd r$, and the only Brownian input of the proof is the finiteness of the exponential moments of this double integral, which we derive from the exponential moments of the maximum of the local time (Lemmas~\ref{lem:sup_local_time} and~\ref{lem:double_integral}). Exponential integrability questions of this type are treated in~\cite[Section~3]{HuNualartSong2011} for noises that are fractional in time with Hurst index larger than $\frac12$, which does not include the time-independent noise of the present limit, and in~\cite{HuHuangNualartTindel2015} for general Gaussian noises.
\item By Theorem~\ref{thm:mean}, when $m=2$ and $c_2>0$, $u(t,x)$ is not integrable, so the limit cannot be identified through its moments, and the convergence in Theorem~\ref{thm:u} is a convergence in law, or in probability on the coupling of Section~\ref{sec:setting}.
\end{enumerate}
\end{remark}

\begin{remark}[On the assumptions]
\label{rem:assumptions}
\begin{enumerate}[label=(\alph*)]
\item We use the assumption that $\Phi$ is a polynomial only in the proof of Lemma~\ref{lem:increments}, where hypercontractivity bounds the moments of the increments of $W_\eps$. It can be replaced by the condition
\[
\sum_{q\ge m}\frac{c_q^2\,p^q}{q!}<\infty\quad\text{for every }p\ge1,
\]
with two changes. First, by Cram\'er's inequality $|H_q(z)|\le k\sqrt{q!}\,e^{z^2/4}$, the Hermite series of $\Phi$ then converges locally uniformly, so that $\Phi$ is continuous and $F_{W_\eps}\in C^2$, as needed in~\eqref{eq:Ito_identity}. Second, in Lemma~\ref{lem:increments}, \eqref{eq:hyper} and Lemma~\ref{lem:second_moments}(i) give
\[
\bigl\|W_\eps(y)-W_\eps(y')\bigr\|_{L^p(\Omega)}\le\Bigl(\frac{2C_R^m}{1-m\alpha}\Bigr)^{1/2}|y-y'|^H\sum_{q\ge m}(p-1)^{q/2}\frac{|c_q|}{\sqrt{q!}},
\]
and the last series is finite by the Cauchy--Schwarz inequality. In the proof of Theorem~\ref{thm:W}(i), the contribution of the chaoses of order $q>m$ is then a series, which tends to $0$ by dominated convergence, since each term is bounded by $c_q^2/q!$ times a constant independent of $q$ and $\eps\le1$.
\item When $m\alpha>1$, the correlations of the potential are integrable, the amplitude $\eps^{-1/2}$ is the relevant one, and the limit is driven by white noise~\cite{PardouxPiatnitski2012,IftimiePardouxPiatnitski2008}. We do not consider this regime, nor the critical case $m\alpha=1$.
\end{enumerate}
\end{remark}

\section{Preliminaries}
\label{sec:prelim}

\subsection{Wiener chaos}

We recall the facts on multiple integrals with respect to $\WW$ that we use; see~\cite[Chapter~1]{Nualart2006} and~\cite[Chapter~2]{NourdinPeccati2012}. For $q\ge1$ and $f\in L^2(\R^q)$, $I_q(f)=I_q(\tilde f)$, where $\tilde f$ is the symmetrisation of $f$, and
\begin{equation}
\label{eq:isometry}
\E\bigl[I_q(f)I_p(h)\bigr]=q!\,\langle\tilde f,\tilde h\rangle_{L^2(\R^q)}\mathbf 1_{\{p=q\}} .
\end{equation}
The closed linear subspace $\mathcal H_q:=\{I_q(f):f\in L^2(\R^q)\}$ of $L^2(\Omega)$ is the $q$-th Wiener chaos. For $h\in L^2(\R)$ with $\|h\|_{L^2(\R)}=1$, $I_1(h)$ is a standard Gaussian variable and~\cite[Proposition~1.1.4]{Nualart2006} (in our normalisation; the Hermite polynomials of~\cite{Nualart2006} are $H_q/q!$)
\begin{equation}
\label{eq:Hq_Iq}
H_q\bigl(I_1(h)\bigr)=I_q\bigl(h^{\otimes q}\bigr).
\end{equation}
The multiple integral $I_q:L^2(\R^q)\to L^2(\Omega)$ is linear and continuous, so it commutes with Bochner integrals: if $z\mapsto f_z$ is continuous from an interval $J$ to $L^2(\R^q)$ and $\mu$ is a finite measure on $J$, then $I_q\bigl(\int_Jf_z\,\mu(\dd z)\bigr)=\int_JI_q(f_z)\,\mu(\dd z)$, where the right-hand side is a Bochner integral in $L^2(\Omega)$. Finally, every element $F$ of $\mathcal H_q$ satisfies the hypercontractivity inequality~\cite[Corollary~2.8.14]{NourdinPeccati2012}
\begin{equation}
\label{eq:hyper}
\bigl(\E|F|^p\bigr)^{1/p}\le(p-1)^{q/2}\bigl(\E|F|^2\bigr)^{1/2},\qquad p\ge2 .
\end{equation}
If $F=\sum_{q=1}^MF_q$ with $F_q\in\mathcal H_q$, then, by Minkowski's inequality,~\eqref{eq:hyper}, the Cauchy--Schwarz inequality and the orthogonality of the chaoses,
\begin{align}
\bigl(\E|F|^p\bigr)^{1/p}
&\le\sum_{q=1}^M\bigl(\E|F_q|^p\bigr)^{1/p}
\le(p-1)^{M/2}\sum_{q=1}^M\bigl(\E|F_q|^2\bigr)^{1/2}\notag\\
&\le(p-1)^{M/2}\sqrt M\Bigl(\sum_{q=1}^M\E|F_q|^2\Bigr)^{1/2}
=(p-1)^{M/2}\sqrt M\bigl(\E|F|^2\bigr)^{1/2}.\label{eq:hyper_sum}
\end{align}

The second chaos has a simple structure, which we use in Section~\ref{sec:mean}.

\begin{lemma}[Second chaos]
\label{lem:second_chaos}
Let $K\in L^2(\R^2)$ be real and symmetric, and let $\mathcal A$ be the Hilbert--Schmidt operator on $L^2(\R)$ with kernel $K$. Let $(\lambda_k)_{k\ge1}$ be its non-zero eigenvalues, repeated according to multiplicity, and $(e_k)$ an orthonormal family of corresponding real eigenfunctions. Then $\xi_k:=I_1(e_k)$ are independent standard Gaussian variables, $\sum_k\lambda_k^2<\infty$, and
\[
I_2(K)=\sum_k\lambda_k\bigl(\xi_k^2-1\bigr)\qquad\text{in }L^2(\Omega).
\]
Moreover, for every $\chi\in L^2(\R)$ with $\|\chi\|_{L^2(\R)}=1$,
\begin{equation}
\label{eq:quad_form}
\langle\mathcal A\chi,\chi\rangle_{L^2(\R)}=\frac12\,\E\Bigl[I_2(K)\,H_2\bigl(I_1(\chi)\bigr)\Bigr].
\end{equation}
\end{lemma}

\begin{proof}
Since $K$ is real, symmetric and square integrable, $\mathcal A$ is a compact self-adjoint operator, and the spectral theorem gives $K=\sum_k\lambda_ke_k\otimes e_k$ in $L^2(\R^2)$, with $\sum_k\lambda_k^2=\|K\|^2_{L^2(\R^2)}<\infty$. The variables $\xi_k=I_1(e_k)$ are centred, jointly Gaussian, with $\E[\xi_k\xi_l]=\langle e_k,e_l\rangle=\mathbf 1_{\{k=l\}}$, hence independent standard Gaussian. By~\eqref{eq:Hq_Iq}, $I_2(e_k\otimes e_k)=H_2(\xi_k)=\xi_k^2-1$, and by continuity of $I_2$,
\[
I_2(K)=\sum_k\lambda_kI_2(e_k\otimes e_k)=\sum_k\lambda_k(\xi_k^2-1)\qquad\text{in }L^2(\Omega).
\]
For the last claim, by~\eqref{eq:Hq_Iq} and~\eqref{eq:isometry},
\begin{align*}
\E\bigl[I_2(K)H_2(I_1(\chi))\bigr]&=\E\bigl[I_2(K)I_2(\chi\otimes\chi)\bigr]
=2\langle K,\chi\otimes\chi\rangle_{L^2(\R^2)}\\
&=2\int_\R\int_\R K(s,s')\chi(s)\chi(s')\,\dd s\,\dd s'
=2\langle\mathcal A\chi,\chi\rangle .
\end{align*}
\end{proof}

\subsection{Hölder seminorms and the Garsia--Rodemich--Rumsey inequality}

For an interval $J\subset\R$, $\theta\in(0,1]$ and a function $w$ on $J$, we write
\[
[w]_{\theta,J}:=\sup_{u,v\in J,\,u\ne v}\frac{|w(u)-w(v)|}{|u-v|^\theta}.
\]
The Garsia--Rodemich--Rumsey inequality~\cite{GRR1970}, with the functions $\Psi(r)=r^p$ and $\varpi(r)=r^{\theta+2/p}$ (see also~\cite[Lemma~A.3.1]{Nualart2006}), states that for every continuous $w$ on a compact interval $J$, every $p\ge1$ and every $\theta>0$,
\begin{equation}
\label{eq:GRR}
[w]_{\theta,J}\le C_{\theta,p}\Bigl(\int_J\int_J\frac{|w(u)-w(v)|^p}{|u-v|^{\theta p+2}}\,\dd u\,\dd v\Bigr)^{1/p},
\qquad
C_{\theta,p}:=8\cdot4^{1/p}\,\frac{\theta+2/p}{\theta},
\end{equation}
where the constant does not depend on $J$: the inequality is proved in~\cite{GRR1970,Nualart2006} for $J=[0,1]$, and the case $J=[c,d]$ follows by applying it to $\tilde w(s)=w(c+(d-c)s)$, since both sides then scale by the factor $(d-c)^\theta$.

\begin{lemma}[Hölder moments from increment moments]
\label{lem:kolmogorov}
Let $X$ be a continuous process on $\R$ with $X(0)=0$, and assume that for some $p\ge1$, $H\in(0,1)$ and $C_X<\infty$,
\[
\E|X(u)-X(v)|^p\le C_X\,|u-v|^{pH}\qquad\text{for all }u,v\in\R .
\]
Let $\theta\in(0,H)$ with $p(H-\theta)>1$. Then, for every $n\ge1$, with $J_n=[-n,n]$,
\[
\E\bigl[[X]_{\theta,J_n}^p\bigr]\le\frac{2^{p(H-\theta)+1}\,C_{\theta,p}^p\,C_X}{p(H-\theta)-1}\,n^{p(H-\theta)},
\qquad
\E\Bigl[\sup_{|y|\le n}|X(y)|^p\Bigr]\le\frac{2^{p(H-\theta)+1}\,C_{\theta,p}^p\,C_X}{p(H-\theta)-1}\,n^{pH}.
\]
\end{lemma}

\begin{proof}
Set $\delta:=p(H-\theta)>1$. By~\eqref{eq:GRR} and Tonelli's theorem,
\begin{align*}
\E\bigl[[X]_{\theta,J_n}^p\bigr]
&\le C_{\theta,p}^p\int_{J_n}\int_{J_n}\frac{\E|X(u)-X(v)|^p}{|u-v|^{\theta p+2}}\,\dd u\,\dd v
\le C_{\theta,p}^pC_X\int_{J_n}\int_{J_n}|u-v|^{\delta-2}\,\dd u\,\dd v\\
&=C_{\theta,p}^pC_X\cdot2\int_0^{2n}(2n-r)\,r^{\delta-2}\,\dd r
\le C_{\theta,p}^pC_X\cdot2\cdot2n\int_0^{2n}r^{\delta-2}\,\dd r\\
&=C_{\theta,p}^pC_X\,\frac{2(2n)^{\delta}}{\delta-1}
=\frac{2^{\delta+1}C_{\theta,p}^pC_X}{\delta-1}\,n^{\delta}.
\end{align*}
Since $X(0)=0$, $|X(y)|\le[X]_{\theta,J_n}|y|^\theta\le[X]_{\theta,J_n}n^\theta$ for $|y|\le n$, and the second bound follows from the first, because $n^{\delta+p\theta}=n^{pH}$.
\end{proof}

\subsection{Compact sets of the weighted space}

\begin{lemma}[Compactness in $\Cr$]
\label{lem:compact}
Let $0<\theta<\rho'<\rho<1$ and $A>0$. The set
\[
\mathcal K_A:=\Bigl\{w\in C(\R):\ w(0)=0,\ [w]_{\theta,J_n}\le A\,n^{\rho'-\theta}\ \text{for all }n\ge1\Bigr\}
\]
is contained in $\Cr$ and is compact in $\Cr$.
\end{lemma}

\begin{proof}
Let $w\in\mathcal K_A$ and $y\in\R$, and let $n:=\max(1,\lceil|y|\rceil)$, so that $n\le1+|y|$. Then
\begin{equation}
\label{eq:K_growth}
|w(y)|=|w(y)-w(0)|\le[w]_{\theta,J_n}|y|^\theta\le A\,n^{\rho'-\theta}n^\theta=A\,n^{\rho'}\le A\,(1+|y|)^{\rho'} .
\end{equation}
Since $\rho'<\rho$, $|w(y)|/(1+|y|)^\rho\le A(1+|y|)^{\rho'-\rho}\to0$, so $w\in\Cr$. Let $(w_k)$ be a sequence in $\mathcal K_A$. On each $J_n$ the functions $w_k$ vanish at $0$ and have Hölder seminorms bounded by $An^{\rho'-\theta}$, so they are uniformly bounded and equicontinuous. By the Arzel\`a--Ascoli theorem and a diagonal argument, there is a subsequence, still denoted $(w_k)$, and a continuous function $w$ such that $w_k\to w$ uniformly on every $J_n$. Passing to the limit in the Hölder bounds shows that $w\in\mathcal K_A$. For every $n\ge1$, by~\eqref{eq:K_growth} applied to $w_k$ and to $w$,
\begin{align*}
\|w_k-w\|_\rho
&\le\sup_{|y|\le n}\frac{|w_k(y)-w(y)|}{(1+|y|)^\rho}+\sup_{|y|>n}\frac{|w_k(y)|+|w(y)|}{(1+|y|)^\rho}\\
&\le\sup_{|y|\le n}|w_k(y)-w(y)|+2A\,(1+n)^{\rho'-\rho}.
\end{align*}
Letting first $k\to\infty$ and then $n\to\infty$ gives $\|w_k-w\|_\rho\to0$. Hence $\mathcal K_A$ is sequentially compact, hence compact, in the metric space $\Cr$.
\end{proof}

\subsection{Brownian motion}

Let $B^*_t:=\sup_{s\le t}|B_s|$. By the reflection principle, $\PP(\sup_{s\le t}B_s>r)=2\PP(B_t>r)\le e^{-r^2/(2t)}$ for $r\ge0$, hence
\begin{equation}
\label{eq:Bstar_tail}
\PP\bigl(B^*_t>r\bigr)\le2\,e^{-r^2/(2t)},\qquad r\ge0 .
\end{equation}

\begin{lemma}[Exponential moments]
\label{lem:exp_Bstar}
For all $t>0$, $x\in\R$, $c>0$ and $r\in(0,2)$,
\[
\E\Bigl[\exp\bigl(c\,(1+|x|+B^*_t)^r\bigr)\Bigr]<\infty .
\]
\end{lemma}

\begin{proof}
Let $a:=1+|x|$. Since $r<2$, $(a+s)^r\le2^r(a^r+s^r)$ for $s\ge0$, and it suffices to show that $\E[\exp(c'(B^*_t)^r)]<\infty$ for $c'=2^rc$. For a non-negative random variable $Y$ and an increasing $C^1$ function $\Theta$ with $\Theta(0)=0$, $\E[e^{\Theta(Y)}]=1+\int_0^\infty\Theta'(s)e^{\Theta(s)}\PP(Y>s)\,\dd s$. With $\Theta(s)=c's^r$ and~\eqref{eq:Bstar_tail},
\[
\E\bigl[e^{c'(B^*_t)^r}\bigr]\le1+2\int_0^\infty c'r\,s^{r-1}\exp\Bigl(c's^r-\frac{s^2}{2t}\Bigr)\dd s<\infty,
\]
because $c's^r-s^2/(2t)\le-s^2/(4t)$ for $s$ large enough, as $r<2$, and $s^{r-1}$ is integrable near $0$.
\end{proof}

\begin{lemma}[Confinement]
\label{lem:confinement}
For all $t>0$ and $w>0$, $\PP\bigl(B^*_t<w\bigr)>0$.
\end{lemma}

\begin{proof}
By Brownian scaling, $\PP(B^*_t<w)=\PP(B^*_{t/w^2}<1)$, so it suffices to show that $\PP(B^*_T<1/4)>0$ for every $T>0$. Since $B$ is continuous with $B_0=0$, $B^*_\delta\to0$ almost surely as $\delta\downarrow0$, and we can fix $\delta\in(0,T]$ with $p:=\PP(B^*_\delta<1/8)>0$. Let $n:=\lceil T/\delta\rceil$ and, for $k=0,\dots,n-1$, let $\tilde B^{(k)}_s:=B_{k\delta+s}-B_{k\delta}$, $s\in[0,\delta]$; these are independent Brownian motions on $[0,\delta]$. Let $E_k$ be the event that $\sup_{s\le\delta}|\tilde B^{(k)}_s|<1/8$ and that $\tilde B^{(k)}_\delta$ has the sign opposite to $B_{k\delta}$ (any sign if $B_{k\delta}=0$). Since $\tilde B^{(k)}$ is independent of $B_{k\delta}$ and its law is invariant under $\tilde B\mapsto-\tilde B$, and $\PP(\tilde B^{(k)}_\delta=0)=0$, we have $\PP(E_k\mid\F^B_{k\delta})\ge p/2$. On $E_0\cap\dots\cap E_{k-1}$, an induction gives $|B_{j\delta}|<1/8$ for $j\le k$: if $|B_{j\delta}|<1/8$ and $E_j$ holds, then $B_{(j+1)\delta}=B_{j\delta}+\tilde B^{(j)}_\delta$ is the sum of two numbers of opposite signs, each of absolute value less than $1/8$, so $|B_{(j+1)\delta}|<1/8$. Moreover, $|B_s|\le|B_{j\delta}|+|\tilde B^{(j)}_{s-j\delta}|<1/4$ for $s\in[j\delta,(j+1)\delta]$. Hence $B^*_{n\delta}<1/4$ on $E_0\cap\dots\cap E_{n-1}$, and, by the Markov property, $\PP(B^*_T<1/4)\ge\PP(E_0\cap\dots\cap E_{n-1})\ge(p/2)^n>0$.
\end{proof}

\paragraph{Local time.} For $x\in\R$, let $L^x_t(y)$ be the local time at level $y$ up to time $t$ of the Brownian motion $x+B$, in its version which is jointly continuous in $(t,y)$; see~\cite[Chapter~VI, Section~1]{RevuzYor1999} for this and the following facts. It satisfies the occupation density formula: almost surely, for every non-negative or bounded measurable $h$ and every $t\ge0$,
\begin{equation}
\label{eq:occupation}
\int_0^th(x+B_s)\,\dd s=\int_\R h(y)\,L^x_t(y)\,\dd y .
\end{equation}
In particular $L^x_t\ge0$, $\int_\R L^x_t(y)\,\dd y=t$, and $L^x_t(y)=0$ for $|y-x|>B^*_t$. For every $\gamma<\frac12$ and every $t$, $y\mapsto L^x_t(y)$ is almost surely $\gamma$-Hölder continuous.

\subsection{Young integrals}

Let $J=[c,d]$ be a compact interval and let $\ell\in C^\gamma(J)$, $w\in C^\theta(J)$ with $\gamma+\theta>1$. The Young integral $\int_J\ell\,\dd w$ is the limit of the Riemann sums $\sum_k\ell(y_k)(w(y_{k+1})-w(y_k))$ along partitions of $J$ whose mesh tends to $0$~\cite{Young1936}, and the Young--Loève inequality~\cite[Theorem~6.8]{FrizVictoir2010} gives, for every partition $\pi=\{c=y_0<\dots<y_N=d\}$,
\begin{equation}
\label{eq:young_loeve}
\Bigl|\int_J\ell\,\dd w-\sum_{k=0}^{N-1}\ell(y_k)\bigl(w(y_{k+1})-w(y_k)\bigr)\Bigr|
\le c_{\gamma,\theta}\,[\ell]_{\gamma,J}\,[w]_{\theta,J}\sum_{k=0}^{N-1}(y_{k+1}-y_k)^{\gamma+\theta},
\end{equation}
where $c_{\gamma,\theta}$ depends only on $\gamma+\theta$. Applied to the trivial partition and to differences, it shows that the map $(\ell,w)\mapsto\int_J\ell\,\dd w$ is continuous on $C^\gamma(J)\times C^\theta(J)$. If $w\in C^1(J)$, the Young integral coincides with the Riemann integral $\int_J\ell(y)w'(y)\,\dd y$.

\section{Proof of Theorem~\ref{thm:W}}
\label{sec:proof_W}

Throughout this section, Assumptions~\ref{ass:g} and~\ref{ass:Phi} hold. Since $\Phi=\sum_{q=m}^Mc_qH_q/q!$, the integrated potential~\eqref{eq:W_eps} decomposes as
\begin{equation}
\label{eq:W_chaos}
W_\eps(y)=\sum_{q=m}^M\frac{c_q}{q!}\,W^{(q)}_\eps(y),
\end{equation}
where
\[
W^{(q)}_\eps(y):=\eps^{-m\alpha/2}\int_0^yH_q\Bigl(g_\eps\Bigl(\frac z\eps\Bigr)\Bigr)\dd z
=\eps^{1-m\alpha/2}\int_0^{y/\eps}H_q\bigl(g_\eps(z)\bigr)\,\dd z .
\]

\subsection{The limiting kernel}

\begin{lemma}
\label{lem:kernel_L2}
Let $m\alpha<1$ and $y\in\R$. Then
\[
\int_{\R^m}\bigl|e_y(\eta_1+\dots+\eta_m)\bigr|^2\prod_{j=1}^m|\eta_j|^{\alpha-1}\,\dd\eta
\le C_\alpha^m\int_0^{|y|}\int_0^{|y|}|u-v|^{-m\alpha}\,\dd u\,\dd v
=\frac{2\,C_\alpha^m\,|y|^{2-m\alpha}}{(1-m\alpha)(2-m\alpha)},
\]
with $C_\alpha:=2(\alpha^{-1}+2)$. In particular, $\widehat{K_0^y}$ defined in~\eqref{eq:K0} belongs to $L^2(\R^m)$.
\end{lemma}

\begin{proof}
Since $|e_{-y}(\zeta)|=|e_y(\zeta)|$, we may assume $y>0$. For $\delta>0$ and $r\in\R$, let
\[
\rho_\delta(r):=\int_\R e^{-ir\eta}|\eta|^{\alpha-1}e^{-\delta\eta^2}\,\dd\eta
=2\int_0^\infty\cos(r\eta)\,\eta^{\alpha-1}e^{-\delta\eta^2}\,\dd\eta .
\]
We first show that $|\rho_\delta(r)|\le C_\alpha|r|^{-\alpha}$ for all $r\ne0$ and $\delta>0$. The substitution $\zeta=|r|\eta$ gives
\[
\rho_\delta(r)=2|r|^{-\alpha}\int_0^\infty\cos(\zeta)\,\zeta^{\alpha-1}e^{-\lambda\zeta^2}\,\dd\zeta,
\qquad\lambda:=\delta/r^2>0 .
\]
The integral over $[0,1]$ is bounded in absolute value by $\int_0^1\zeta^{\alpha-1}\,\dd\zeta=\alpha^{-1}$. On $[1,\infty)$, the function $k(\zeta)=\zeta^{\alpha-1}e^{-\lambda\zeta^2}$ is positive and non-increasing, with $k(1)\le1$. By the second mean value theorem, for every $N>1$ there is $\xi\in[1,N]$ such that
\[
\int_1^N\cos(\zeta)\,k(\zeta)\,\dd\zeta=k(1)\int_1^\xi\cos(\zeta)\,\dd\zeta=k(1)\bigl(\sin\xi-\sin1\bigr),
\]
whose absolute value is at most $2$. Letting $N\to\infty$, the integral over $[1,\infty)$ is bounded by $2$ in absolute value, and the claim follows.

Now let
\[
I_\delta:=\int_{\R^m}|e_y(\eta_1+\dots+\eta_m)|^2\prod_{j=1}^m|\eta_j|^{\alpha-1}e^{-\delta\eta_j^2}\,\dd\eta .
\]
Writing $|e_y(\zeta)|^2=\int_0^y\int_0^ye^{-i(u-v)\zeta}\,\dd u\,\dd v$, and noting that the integrand is absolutely integrable on $[0,y]^2\times\R^m$ (it is bounded by $\prod_j|\eta_j|^{\alpha-1}e^{-\delta\eta_j^2}$), Fubini's theorem gives
\[
I_\delta=\int_0^y\int_0^y\prod_{j=1}^m\Bigl(\int_\R e^{-i(u-v)\eta_j}|\eta_j|^{\alpha-1}e^{-\delta\eta_j^2}\,\dd\eta_j\Bigr)\dd u\,\dd v
=\int_0^y\int_0^y\rho_\delta(u-v)^m\,\dd u\,\dd v .
\]
Since $I_\delta\ge0$, the bound on $\rho_\delta$ yields
\[
I_\delta\le\int_0^y\int_0^y|\rho_\delta(u-v)|^m\,\dd u\,\dd v\le C_\alpha^m\int_0^y\int_0^y|u-v|^{-m\alpha}\,\dd u\,\dd v .
\]
As $\delta\downarrow0$, $I_\delta$ increases to the integral in the statement, by monotone convergence. Finally,
\[
\int_0^y\int_0^y|u-v|^{-m\alpha}\,\dd u\,\dd v=2\int_0^y(y-r)\,r^{-m\alpha}\,\dd r
=2\Bigl(\frac{y^{2-m\alpha}}{1-m\alpha}-\frac{y^{2-m\alpha}}{2-m\alpha}\Bigr)
=\frac{2\,y^{2-m\alpha}}{(1-m\alpha)(2-m\alpha)} .
\]
\end{proof}

\subsection{Chaos representation}

\begin{lemma}
\label{lem:chaos_rep}
Let $q\ge1$, $\eps>0$ and $y\in\R$. Then $W^{(q)}_\eps(y)=I_q\bigl(K^{y,(q)}_\eps\bigr)$ almost surely, where $K^{y,(q)}_\eps\in L^2(\R^q)$ is symmetric with Fourier transform
\begin{equation}
\label{eq:Keps_hat}
\widehat{K^{y,(q)}_\eps}(\eta)=\eps^{(q-m\alpha)/2}\,e_y(\eta_1+\dots+\eta_q)\prod_{j=1}^q\sqrt{2\pi f(\eps\eta_j)} .
\end{equation}
For $q=m$, this reads
\begin{equation}
\label{eq:Keps_hat_m}
\widehat{K^{y,(m)}_\eps}(\eta)=(2\pi)^{m/2}\,e_y(\eta_1+\dots+\eta_m)\prod_{j=1}^m\psi(\eta_j)\sqrt{b(\eps\eta_j)} .
\end{equation}
\end{lemma}

\begin{proof}
For $h\in L^2(\R)$, $\int_\R h(s)\,\dd\WW^\eps(s)=\eps^{-1/2}\int_\R h(v/\eps)\,\dd\WW(v)$, as one checks on indicator functions: if $h=\mathbf 1_{[c,d]}$, both sides equal $\eps^{-1/2}(\WW(\eps d)-\WW(\eps c))$. Hence, by~\eqref{eq:g_eps},
\[
g_\eps(z)=I_1(\mathsf a_{\eps,z}),\qquad\mathsf a_{\eps,z}(v):=\eps^{-1/2}\,\mathsf a\Bigl(z-\frac v\eps\Bigr),
\qquad
\|\mathsf a_{\eps,z}\|_{L^2(\R)}^2=\int_\R\mathsf a(z-w)^2\,\dd w=1 .
\]
By~\eqref{eq:Hq_Iq}, $H_q(g_\eps(z))=I_q(\mathsf a_{\eps,z}^{\otimes q})$. The map $z\mapsto\mathsf a_{\eps,z}^{\otimes q}$ is continuous from $\R$ to $L^2(\R^q)$, by continuity of translations in $L^2$. Since $z\mapsto H_q(g_\eps(z))$ is continuous, its Riemann integral over $[0,y/\eps]$ is the almost sure and $L^2(\Omega)$ limit of Riemann sums, and the commutation of $I_q$ with Bochner integrals gives
\[
W^{(q)}_\eps(y)=\eps^{1-m\alpha/2}\int_0^{y/\eps}I_q\bigl(\mathsf a_{\eps,z}^{\otimes q}\bigr)\,\dd z
=I_q\bigl(K^{y,(q)}_\eps\bigr),
\qquad
K^{y,(q)}_\eps:=\eps^{1-m\alpha/2}\int_0^{y/\eps}\mathsf a_{\eps,z}^{\otimes q}\,\dd z .
\]
The kernel is symmetric. To compute its Fourier transform, note that, since $\mathsf a$ is even, with the substitutions $v=\eps v'$ and $w=z-v'$,
\[
\widehat{\mathsf a_{\eps,z}}(\eta)
=\eps^{-1/2}\int_\R e^{-iv\eta}\mathsf a\Bigl(z-\frac v\eps\Bigr)\dd v
=\eps^{1/2}\int_\R e^{-i\eps(z-w)\eta}\mathsf a(w)\,\dd w
=\eps^{1/2}e^{-i\eps z\eta}\,\hat{\mathsf a}(\eps\eta).
\]
The Fourier transform is a continuous linear map on $L^2(\R^q)$, so it commutes with the Bochner integral, and
\begin{align*}
\widehat{K^{y,(q)}_\eps}(\eta)
&=\eps^{1-m\alpha/2}\int_0^{y/\eps}\eps^{q/2}e^{-i\eps z(\eta_1+\dots+\eta_q)}\prod_{j=1}^q\hat{\mathsf a}(\eps\eta_j)\,\dd z\\
&=\eps^{1-m\alpha/2+q/2}\prod_{j=1}^q\hat{\mathsf a}(\eps\eta_j)\cdot\frac1\eps\int_0^ye^{-iu(\eta_1+\dots+\eta_q)}\,\dd u\\
&=\eps^{(q-m\alpha)/2}\,e_y(\eta_1+\dots+\eta_q)\prod_{j=1}^q\sqrt{2\pi f(\eps\eta_j)},
\end{align*}
where we substituted $u=\eps z$ and used $\hat{\mathsf a}=\sqrt{2\pi f}$. This is~\eqref{eq:Keps_hat}. For $q=m$, by~\ref{G1},
\[
\prod_{j=1}^m\sqrt{f(\eps\eta_j)}=\prod_{j=1}^m\bigl(\eps^{(\alpha-1)/2}|\eta_j|^{(\alpha-1)/2}\sqrt{b(\eps\eta_j)}\bigr)
=\eps^{m(\alpha-1)/2}\prod_{j=1}^m\psi(\eta_j)\sqrt{b(\eps\eta_j)},
\]
and $\eps^{(m-m\alpha)/2}\eps^{m(\alpha-1)/2}=1$, which gives~\eqref{eq:Keps_hat_m}.
\end{proof}

\subsection{Second moments}

\begin{lemma}
\label{lem:second_moments}
Let $q\in\{m,\dots,M\}$, $\eps>0$ and $y,y'\in\R$.
\begin{enumerate}[label=\textup{(\roman*)}]
\item $\displaystyle\E\bigl|W^{(q)}_\eps(y)-W^{(q)}_\eps(y')\bigr|^2\le\frac{2\,q!\,C_R^m}{1-m\alpha}\,|y-y'|^{2H}$.
\item $\displaystyle\lim_{\eps\to0}\E\bigl|W^{(m)}_\eps(y)\bigr|^2=\sigma_m^2\,|y|^{2H}$, with $\sigma_m^2=m!\kappa^m/(H(2H-1))$.
\item If $q>m$, then $\displaystyle\lim_{\eps\to0}\E\bigl|W^{(q)}_\eps(y)\bigr|^2=0$.
\end{enumerate}
\end{lemma}

\begin{proof}
For $z,z'\in\R$, $(g_\eps(z),g_\eps(z'))$ is a centred Gaussian vector with unit variances and covariance $R(z-z')$, so that $\E[H_q(g_\eps(z))H_q(g_\eps(z'))]=q!\,R(z-z')^q$~\cite[Lemma~1.1.1]{Nualart2006}, in our normalisation of the Hermite polynomials. Assume $y'<y$ and let $\delta:=y-y'$. By Fubini's theorem and stationarity,
\begin{align}
\E\bigl|W^{(q)}_\eps(y)-W^{(q)}_\eps(y')\bigr|^2
&=\eps^{2-m\alpha}\int_{y'/\eps}^{y/\eps}\int_{y'/\eps}^{y/\eps}q!\,R(z-z')^q\,\dd z\,\dd z'\notag\\
&=q!\,\eps^{2-m\alpha}\int_0^{\delta/\eps}\int_0^{\delta/\eps}R(z-z')^q\,\dd z\,\dd z' .\label{eq:var_incr}
\end{align}

(i) Since $|R|\le R(0)=1$ and $q\ge m$, $|R|^q\le|R|^m\le C_R^m(1+|r|)^{-m\alpha}$ by~\ref{G2}. For $T>0$,
\[
\int_0^T\int_0^T(1+|z-z'|)^{-m\alpha}\,\dd z\,\dd z'
=2\int_0^T(T-r)(1+r)^{-m\alpha}\,\dd r
\le2T\int_0^Tr^{-m\alpha}\,\dd r=\frac{2\,T^{2-m\alpha}}{1-m\alpha} .
\]
With $T=\delta/\eps$ and $2-m\alpha=2H$, \eqref{eq:var_incr} gives
\[
\E\bigl|W^{(q)}_\eps(y)-W^{(q)}_\eps(y')\bigr|^2
\le q!\,C_R^m\,\eps^{2H}\,\frac{2(\delta/\eps)^{2H}}{1-m\alpha}
=\frac{2\,q!\,C_R^m}{1-m\alpha}\,\delta^{2H}.
\]

(ii) Take $y'=0$ and assume $y>0$ (the case $y<0$ is symmetric). The substitution $z=u/\eps$, $z'=v/\eps$ in~\eqref{eq:var_incr} gives
\[
\E\bigl|W^{(m)}_\eps(y)\bigr|^2=m!\int_0^y\int_0^y\eps^{-m\alpha}R\Bigl(\frac{u-v}\eps\Bigr)^m\dd u\,\dd v .
\]
For $u\ne v$, $\eps^{-m\alpha}R((u-v)/\eps)^m=|u-v|^{-m\alpha}\bigl(|\tfrac{u-v}\eps|^{\alpha}R(\tfrac{u-v}\eps)\bigr)^m\to\kappa^m|u-v|^{-m\alpha}$ by~\ref{G2}, and
\[
\Bigl|\eps^{-m\alpha}R\Bigl(\frac{u-v}\eps\Bigr)^m\Bigr|\le C_R^m\,\eps^{-m\alpha}\Bigl(1+\frac{|u-v|}\eps\Bigr)^{-m\alpha}\le C_R^m\,|u-v|^{-m\alpha},
\]
which is integrable on $[0,y]^2$ because $m\alpha<1$. By dominated convergence and the computation at the end of the proof of Lemma~\ref{lem:kernel_L2},
\[
\lim_{\eps\to0}\E\bigl|W^{(m)}_\eps(y)\bigr|^2=m!\,\kappa^m\,\frac{2\,y^{2-m\alpha}}{(1-m\alpha)(2-m\alpha)}=\frac{m!\,\kappa^m}{H(2H-1)}\,y^{2H},
\]
since $2-m\alpha=2H$ and $1-m\alpha=2H-1$.

(iii) Let $q\ge m+1$ and $T:=|y|/\eps$. Now $|R|^q\le|R|^{m+1}\le C_R^{m+1}(1+|r|)^{-(m+1)\alpha}$, and as in (i),
\[
\E\bigl|W^{(q)}_\eps(y)\bigr|^2\le2\,q!\,C_R^{m+1}\,\eps^{2-m\alpha}\,T\int_0^T(1+r)^{-(m+1)\alpha}\,\dd r .
\]
We distinguish three cases.
\begin{itemize}
\item If $(m+1)\alpha<1$, the integral is at most $T^{1-(m+1)\alpha}/(1-(m+1)\alpha)$, and the bound is a constant times $\eps^{2-m\alpha}T^{2-(m+1)\alpha}=\eps^{\alpha}|y|^{2-(m+1)\alpha}$, which tends to $0$.
\item If $(m+1)\alpha=1$, the integral equals $\log(1+T)$, and the bound is a constant times $\eps^{2-m\alpha}T\log(1+T)=\eps^{1-m\alpha}|y|\log(1+|y|/\eps)$, which tends to $0$ because $1-m\alpha>0$.
\item If $(m+1)\alpha>1$, the integral is at most $1/((m+1)\alpha-1)$, and the bound is a constant times $\eps^{2-m\alpha}T=\eps^{1-m\alpha}|y|$, which tends to $0$.
\end{itemize}
\end{proof}

\begin{lemma}[Moments of increments]
\label{lem:increments}
For every $p\ge2$ there is a constant $D_p$, which does not depend on $\eps$, such that for all $\eps>0$ and $y,y'\in\R$,
\[
\E\bigl|W_\eps(y)-W_\eps(y')\bigr|^p\le D_p\,|y-y'|^{pH} .
\]
\end{lemma}

\begin{proof}
By~\eqref{eq:W_chaos} and Lemma~\ref{lem:chaos_rep}, $W_\eps(y)-W_\eps(y')=\sum_{q=m}^MF_q$ with $F_q:=\frac{c_q}{q!}(W^{(q)}_\eps(y)-W^{(q)}_\eps(y'))\in\mathcal H_q$. By orthogonality of the chaoses and Lemma~\ref{lem:second_moments}(i),
\[
\E\bigl|W_\eps(y)-W_\eps(y')\bigr|^2=\sum_{q=m}^M\frac{c_q^2}{(q!)^2}\,\E\bigl|W^{(q)}_\eps(y)-W^{(q)}_\eps(y')\bigr|^2
\le\Bigl(\sum_{q=m}^M\frac{c_q^2}{q!}\Bigr)\frac{2\,C_R^m}{1-m\alpha}\,|y-y'|^{2H} .
\]
The claim follows from~\eqref{eq:hyper_sum}, with
\[
D_p:=\bigl((p-1)^{M/2}\sqrt M\bigr)^p\Bigl(\frac{2\,C_R^m}{1-m\alpha}\sum_{q=m}^M\frac{c_q^2}{q!}\Bigr)^{p/2}.
\]
\end{proof}

\subsection{The Hermite process}

\begin{proposition}
\label{prop:Z}
For every $y\in\R$, $W^{(m)}_\eps(y)\to\sigma_mZ(y)$ in $L^2(\Omega)$ as $\eps\to0$. The process $Z$ of Definition~\ref{def:hermite} takes values in the $m$-th Wiener chaos $\mathcal H_m$, is self-similar with index $H$, has stationary increments, and
\begin{equation}
\label{eq:Z_cov}
\E\bigl[Z(y)Z(y')\bigr]=\frac12\bigl(|y|^{2H}+|y'|^{2H}-|y-y'|^{2H}\bigr),\qquad
\E|Z(y)-Z(y')|^p\le(p-1)^{pm/2}|y-y'|^{pH}
\end{equation}
for all $y,y'\in\R$ and $p\ge2$. In particular, $Z$ has a continuous version, which is locally Hölder continuous of every order $\theta<H$, and $\E[[Z]^p_{\theta,J}]<\infty$ for every compact interval $J$, $\theta<H$ and $p\ge1$.
\end{proposition}

\begin{proof}
By Lemma~\ref{lem:chaos_rep} and~\eqref{eq:isometry}, since $K^{y,(m)}_\eps$ and $K_0^y$ are symmetric,
\[
\E\bigl|W^{(m)}_\eps(y)-I_m(K_0^y)\bigr|^2=m!\,\bigl\|K^{y,(m)}_\eps-K_0^y\bigr\|^2_{L^2(\R^m)}
=\frac{m!}{(2\pi)^m}\Bigl\|\widehat{K^{y,(m)}_\eps}-\widehat{K_0^y}\Bigr\|^2_{L^2(\R^m)} .
\]
By~\eqref{eq:Keps_hat_m}, \eqref{eq:K0} and the continuity of $b$ at $0$, $\widehat{K^{y,(m)}_\eps}(\eta)\to\widehat{K_0^y}(\eta)$ for every $\eta$ with non-zero coordinates, and
\[
\Bigl|\widehat{K^{y,(m)}_\eps}(\eta)-\widehat{K_0^y}(\eta)\Bigr|^2\le(2\pi)^m\bigl(\|b\|_\infty^{m/2}+b(0)^{m/2}\bigr)^2\,|e_y(\eta_1+\dots+\eta_m)|^2\prod_{j=1}^m|\eta_j|^{\alpha-1},
\]
which is integrable by Lemma~\ref{lem:kernel_L2}. By dominated convergence, $W^{(m)}_\eps(y)\to I_m(K_0^y)$ in $L^2(\Omega)$. By Lemma~\ref{lem:second_moments}(ii), $\E|I_m(K_0^y)|^2=\sigma_m^2|y|^{2H}$, and $I_m(K_0^y)=\sigma_mZ(y)$ by Definition~\ref{def:hermite}. This proves the first claim, and $Z(y)\in\mathcal H_m$.

For $y,y'\in\R$, the process $z\mapsto g_\eps(z)$ is stationary, so $W^{(m)}_\eps(y)-W^{(m)}_\eps(y')$ has the same law as $\pm W^{(m)}_\eps(|y-y'|)$. Passing to the limit in $L^2$,
\[
\E|Z(y)-Z(y')|^2=\lim_{\eps\to0}\sigma_m^{-2}\E\bigl|W^{(m)}_\eps(|y-y'|)\bigr|^2=|y-y'|^{2H},
\]
and the covariance in~\eqref{eq:Z_cov} follows by polarisation, since $Z(0)=0$. The moment bound follows from~\eqref{eq:hyper}, since $Z(y)-Z(y')\in\mathcal H_m$.

Finally, since $g_\eps$ has the same law as $g$, the process $(W^{(m)}_\eps(y))_{y\in\R}$ has the same law as $(\eps^HS(y/\eps))_{y\in\R}$, where $S(T):=\int_0^TH_m(g(z))\,\dd z$. Hence, for $c>0$,
\[
\bigl(W^{(m)}_\eps(cy)\bigr)_y\overset{d}{=}\bigl(\eps^HS(cy/\eps)\bigr)_y=\bigl(c^H(\eps/c)^HS(y/(\eps/c))\bigr)_y\overset{d}{=}\bigl(c^HW^{(m)}_{\eps/c}(y)\bigr)_y,
\]
and, by stationarity of $g$, for $h\in\R$, $(W^{(m)}_\eps(y+h)-W^{(m)}_\eps(h))_y\overset{d}{=}(W^{(m)}_\eps(y))_y$. Passing to the $L^2$ limit at finitely many points shows that $Z$ is self-similar with index $H$ and has stationary increments. The existence of a continuous version and the local Hölder continuity follow from Kolmogorov's criterion, and the bound on $\E[[Z]^p_{\theta,J}]$ from Lemma~\ref{lem:kolmogorov}, applied with $p$ large enough that $p(H-\theta)>1$, together with Jensen's inequality for smaller $p$.
\end{proof}

\subsection{Proof of Theorem~\ref{thm:W}}

\emph{Part (i).} By~\eqref{eq:W_chaos}, Proposition~\ref{prop:Z} and Lemma~\ref{lem:second_moments}(iii),
\[
W_\eps(y)=\frac{c_m}{m!}W^{(m)}_\eps(y)+\sum_{q=m+1}^M\frac{c_q}{q!}W^{(q)}_\eps(y)\longrightarrow\frac{c_m}{m!}\,\sigma_m\,Z(y)\qquad\text{in }L^2(\Omega),
\]
and $\frac{c_m}{m!}\sigma_m=c_m\sqrt{\kappa^m/(m!\,H(2H-1))}=\beta$.

\emph{Part (ii).} Fix $\rho\in(H,1)$ and choose $\rho'\in(H,\rho)$, $\theta\in(0,H)$ and $p\ge2$ such that $p(H-\theta)>1$ and $p(\rho'-H)>1$. By Lemmas~\ref{lem:increments} and~\ref{lem:kolmogorov}, applied to the continuous process $W_\eps$, which vanishes at $0$, there is a constant $D$, independent of $\eps$ and $n$, such that $\E[[W_\eps]^p_{\theta,J_n}]\le D\,n^{p(H-\theta)}$ for all $n\ge1$. For $A>0$, let $\mathcal K_A$ be the compact subset of $\Cr$ of Lemma~\ref{lem:compact}. By Markov's inequality,
\begin{align*}
\PP\bigl(W_\eps\notin\mathcal K_A\bigr)
&\le\sum_{n\ge1}\PP\bigl([W_\eps]_{\theta,J_n}>A\,n^{\rho'-\theta}\bigr)
\le\sum_{n\ge1}\frac{D\,n^{p(H-\theta)}}{A^p\,n^{p(\rho'-\theta)}}\\
&=\frac D{A^p}\sum_{n\ge1}n^{-p(\rho'-H)}=:\frac{D'}{A^p},
\end{align*}
where $D'<\infty$ because $p(\rho'-H)>1$. Letting $A\to\infty$ shows that $W_\eps\in\bigcup_{A}\mathcal K_A\subset\Cr$ almost surely, and that the laws of $(W_\eps)_{\eps>0}$ on $\Cr$ are tight. The same argument, based on~\eqref{eq:Z_cov}, shows that $\beta Z\in\Cr$ almost surely. The Borel $\sigma$-field of the separable Banach space $\Cr$ is generated by the evaluation maps $w\mapsto w(y)$, $y\in\mathbb Q$, because the norm is the supremum over rational $y$ of $|w(y)|/(1+|y|)^\rho$, by continuity. Hence a probability measure on $\Cr$ is determined by its finite-dimensional distributions. By part~(i), the finite-dimensional distributions of $W_\eps$ converge to those of $\beta Z$. Together with tightness and Prokhorov's theorem, this proves that $W_\eps\to\beta Z$ in law in $\Cr$. \qed

\section{Proof of Theorem~\ref{thm:u}}
\label{sec:proof_u}

Throughout this section, $\rho\in(0,1)$, $T>0$ and $\varphi$ is bounded and continuous. All expectations are with respect to $B$, the function $w\in\Cr$ being deterministic.

\subsection{An exponential bound}

\begin{lemma}
\label{lem:exp_bound}
Let $w\in C(\R)$ satisfy $|w(y)|\le K(1+|y|)^\rho$ for all $y\in\R$, for some $K\ge0$. Let $t\in[0,T]$, $x\in\R$, $\lambda\ge0$, and $S:=1+|x|+B^*_T$. Then, almost surely,
\begin{equation}
\label{eq:Lambda_pathwise}
\bigl|\Lambda_t^x(w)\bigr|\le4K\,S^{1+\rho}+2\,|M_t|,\qquad M_t:=\int_0^tw(x+B_s)\,\dd B_s,
\end{equation}
and
\begin{equation}
\label{eq:exp_bound}
\E^B\Bigl[e^{\lambda|\Lambda_t^x(w)|}\Bigr]\le\Gamma(\lambda K,x):=\sqrt2\,\Bigl(\E^B\bigl[e^{8\lambda KS^{1+\rho}}\bigr]\Bigr)^{1/2}\Bigl(\E^B\bigl[e^{32\lambda^2K^2TS^{2\rho}}\bigr]\Bigr)^{1/4}<\infty .
\end{equation}
The function $\Gamma$ is non-decreasing in $\lambda K$ and in $|x|$.
\end{lemma}

\begin{proof}
For every $y\in\R$,
\[
|F_w(y)|\le\int_0^{|y|}K(1+r)^\rho\,\dd r\le\frac{K(1+|y|)^{1+\rho}}{1+\rho}\le K(1+|y|)^{1+\rho}.
\]
Since $|x+B_t|\le|x|+B^*_T$ and $|x|\le|x|+B^*_T$, both $|F_w(x+B_t)|$ and $|F_w(x)|$ are at most $KS^{1+\rho}$, and~\eqref{eq:Lambda_pathwise} follows from~\eqref{eq:Lambda}. By the Cauchy--Schwarz inequality,
\[
\E^B\bigl[e^{\lambda|\Lambda_t^x(w)|}\bigr]\le\E^B\bigl[e^{4\lambda KS^{1+\rho}}e^{2\lambda|M_t|}\bigr]
\le\Bigl(\E^B\bigl[e^{8\lambda KS^{1+\rho}}\bigr]\Bigr)^{1/2}\Bigl(\E^B\bigl[e^{4\lambda|M_t|}\bigr]\Bigr)^{1/2}.
\]
The quadratic variation of the continuous local martingale $M$ satisfies $\langle M\rangle_t=\int_0^tw(x+B_s)^2\,\dd s\le TK^2S^{2\rho}$. For $\mu\in\R$, the process $\exp(2\mu M-2\mu^2\langle M\rangle)$ is a positive local martingale, hence a supermartingale with expectation at most $1$, and the Cauchy--Schwarz inequality gives
\begin{align*}
\E^B\bigl[e^{\mu M_t}\bigr]&=\E^B\bigl[e^{\mu M_t-\mu^2\langle M\rangle_t}e^{\mu^2\langle M\rangle_t}\bigr]\\
&\le\Bigl(\E^B\bigl[e^{2\mu M_t-2\mu^2\langle M\rangle_t}\bigr]\Bigr)^{1/2}\Bigl(\E^B\bigl[e^{2\mu^2\langle M\rangle_t}\bigr]\Bigr)^{1/2}
\le\Bigl(\E^B\bigl[e^{2\mu^2TK^2S^{2\rho}}\bigr]\Bigr)^{1/2}.
\end{align*}
Applying this with $\mu=\pm4\lambda$ and using $e^{4\lambda|M_t|}\le e^{4\lambda M_t}+e^{-4\lambda M_t}$,
\[
\E^B\bigl[e^{4\lambda|M_t|}\bigr]\le2\Bigl(\E^B\bigl[e^{32\lambda^2K^2TS^{2\rho}}\bigr]\Bigr)^{1/2},
\]
which gives~\eqref{eq:exp_bound}. Both expectations in $\Gamma$ are finite by Lemma~\ref{lem:exp_Bstar}, since $1+\rho<2$ and $2\rho<2$. The monotonicity is clear.
\end{proof}

\begin{lemma}
\label{lem:U_well_defined}
For every $w\in\Cr$, $t\in[0,T]$ and $x\in\R$, $|\U_{t,x}(w)|\le\|\varphi\|_\infty\,\Gamma(\|w\|_\rho,x)<\infty$. Moreover, for every $\eps>0$, every realisation of the potential such that $W_\eps\in\Cr$, and all $t\ge0$, $x\in\R$, the identity~\eqref{eq:Ito_identity} holds $\PP^B$-almost surely; in particular $u_\eps(t,x)=\U_{t,x}(W_\eps)$ is finite. By Theorem~\ref{thm:W}(ii), this holds for almost every realisation of the potential.
\end{lemma}

\begin{proof}
The first claim follows from Lemma~\ref{lem:exp_bound} with $\lambda=1$ and $K=\|w\|_\rho$. For the second, fix a realisation of the potential. Then $F_{W_\eps}$ is a deterministic function of class $C^2$, with $F_{W_\eps}'=W_\eps$ and $F_{W_\eps}''=V_\eps$, and Itô's formula for the Brownian motion $x+B$ gives~\eqref{eq:Ito_identity}, $\PP^B$-almost surely. Taking $\E^B$ in the exponential form of~\eqref{eq:Ito_identity} gives $u_\eps(t,x)=\U_{t,x}(W_\eps)$.
\end{proof}

\subsection{Continuity of the Feynman--Kac functional}

\begin{lemma}
\label{lem:U_continuous}
The map $(t,x,w)\mapsto\U_{t,x}(w)$ is continuous on $[0,T]\times\R\times\Cr$. More precisely, if $(t_n,x_n,w_n)\to(t,x,w)$ in $[0,T]\times\R\times\Cr$, then $\Lambda_{t_n}^{x_n}(w_n)\to\Lambda_t^x(w)$ in $L^2(\PP^B)$, and $\U_{t_n,x_n}(w_n)\to\U_{t,x}(w)$.
\end{lemma}

\begin{proof}
Let $K:=\sup_n\|w_n\|_\rho\vee\|w\|_\rho$ and $X:=\sup_n|x_n|\vee|x|$, which are finite, and let $\bar S:=1+X+B^*_T$. Write $\Lambda_n:=\Lambda_{t_n}^{x_n}(w_n)$, $\Lambda:=\Lambda_t^x(w)$ and $a_n:=x_n+B_{t_n}$, $a:=x+B_t$, so that $a_n\to a$ almost surely and $|a_n|,|a|\le\bar S$.

\emph{The primitive terms.} For $|y|\le\bar S-1$, so that $1+|y|\le\bar S$,
\[
|F_{w_n}(y)-F_w(y)|\le\int_0^{|y|}|w_n(z)-w(z)|\,\dd z\le\|w_n-w\|_\rho\,\bar S^{1+\rho},
\]
so that
\[
|F_{w_n}(a_n)-F_w(a)|\le\|w_n-w\|_\rho\,\bar S^{1+\rho}+|F_w(a_n)-F_w(a)|\longrightarrow0\quad\text{almost surely},
\]
by continuity of $F_w$. The same argument gives $F_{w_n}(x_n)\to F_w(x)$. Moreover, all these terms are bounded by $K\bar S^{1+\rho}$, which belongs to $L^2(\PP^B)$ by Lemma~\ref{lem:exp_Bstar}; by dominated convergence, the convergence also holds in $L^2(\PP^B)$.

\emph{The stochastic integrals.} By the Itô isometry,
\begin{align*}
&\E^B\Bigl|\int_0^{t_n}w_n(x_n+B_s)\,\dd B_s-\int_0^tw(x+B_s)\,\dd B_s\Bigr|^2\\
&\qquad=\E^B\int_0^T\bigl|\mathbf 1_{\{s\le t_n\}}w_n(x_n+B_s)-\mathbf 1_{\{s\le t\}}w(x+B_s)\bigr|^2\,\dd s .
\end{align*}
For every $s\ne t$ and every $\omega$, the integrand tends to $0$: $\mathbf 1_{\{s\le t_n\}}\to\mathbf 1_{\{s\le t\}}$, and $w_n(x_n+B_s)\to w(x+B_s)$ because $w_n\to w$ uniformly on $[-\bar S,\bar S]$ and $w$ is continuous. The integrand is bounded by $4K^2\bar S^{2\rho}$, which is integrable on $[0,T]\times\Omega^B$. By dominated convergence, the right-hand side tends to $0$. Hence $\Lambda_n\to\Lambda$ in $L^2(\PP^B)$.

\emph{The exponential.} By Lemma~\ref{lem:exp_bound} with $\lambda=2$, and since $\Gamma$ is non-decreasing,
\[
\sup_n\E^B\bigl[e^{2\Lambda_n}\bigr]\le\Gamma(2K,X)<\infty,
\]
so that $(e^{\Lambda_n})_n$ is uniformly integrable. Since $\Lambda_n\to\Lambda$ in probability, $e^{\Lambda_n}\to e^\Lambda$ in probability, hence in $L^1(\PP^B)$. Finally, $\varphi(a_n)\to\varphi(a)$ almost surely and $|\varphi(a_n)|\le\|\varphi\|_\infty$, so
\[
\bigl|\U_{t_n,x_n}(w_n)-\U_{t,x}(w)\bigr|
\le\|\varphi\|_\infty\,\E^B\bigl|e^{\Lambda_n}-e^\Lambda\bigr|+\E^B\bigl[|\varphi(a_n)-\varphi(a)|\,e^\Lambda\bigr]\longrightarrow0,
\]
the last term by dominated convergence.
\end{proof}

\subsection{Conclusion of the proof}

Fix $\rho\in(H,1)$, so that Theorem~\ref{thm:W}(ii) applies. Let $\mathcal C:=C([0,T]\times\R)$, with the topology of uniform convergence on compact sets, which is Polish. By Lemma~\ref{lem:U_continuous}, for every $w\in\Cr$, the function $\Psi(w):(t,x)\mapsto\U_{t,x}(w)$ belongs to $\mathcal C$. We claim that $\Psi:\Cr\to\mathcal C$ is continuous. Let $w_n\to w$ in $\Cr$, fix $k\ge1$ and suppose that $\sup_{[0,T]\times[-k,k]}|\Psi(w_n)-\Psi(w)|$ does not tend to $0$. Then there are $\eta>0$, a subsequence $(n_j)$ and points $(t_j,x_j)\in[0,T]\times[-k,k]$ with $|\U_{t_j,x_j}(w_{n_j})-\U_{t_j,x_j}(w)|\ge\eta$. By compactness, we may assume that $(t_j,x_j)\to(t,x)$. By Lemma~\ref{lem:U_continuous}, both $\U_{t_j,x_j}(w_{n_j})$ and $\U_{t_j,x_j}(w)$ tend to $\U_{t,x}(w)$, a contradiction.

By Lemma~\ref{lem:U_well_defined}, $u_\eps=\Psi(W_\eps)$ almost surely, and $u=\Psi(\beta Z)$ by definition; in particular both have continuous versions, and $u$ is well defined because $\beta Z\in\Cr$ almost surely. By Theorem~\ref{thm:W}(ii) and the continuous mapping theorem, $u_\eps=\Psi(W_\eps)\to\Psi(\beta Z)=u$ in law in $\mathcal C$. The convergence of the finite-dimensional distributions follows, since evaluation at finitely many points is continuous on $\mathcal C$.

For the convergence in probability, consider the pair $(W_\eps,\beta Z)$, defined on the same probability space. Since each component is tight in $\Cr$, the pair is tight in $\Cr\times\Cr$. Let $(X,Y)$ be the limit in law along a subsequence. For every rational $y$, $W_\eps(y)-\beta Z(y)\to0$ in $L^2(\Omega)$ by Theorem~\ref{thm:W}(i), so that $X(y)-Y(y)=0$ almost surely. By continuity, $X=Y$ almost surely. Hence every limit point of the law of $W_\eps-\beta Z$ in $\Cr$ is the Dirac mass at $0$, that is, $\|W_\eps-\beta Z\|_\rho\to0$ in probability. Since $\Psi$ is continuous, $(\Psi(W_\eps),\Psi(\beta Z))$ converges in law, along every subsequence, to $(\Psi(X),\Psi(X))$, so that the distance between $u_\eps=\Psi(W_\eps)$ and $u=\Psi(\beta Z)$ in $\mathcal C$ converges in law, hence in probability, to $0$.

Finally, let $\varphi\ge0$ with $\varphi\not\equiv0$, and $t\in(0,T]$, $x\in\R$. For $w\in\Cr$, $\U_{t,x}(w)<\infty$ by Lemma~\ref{lem:U_well_defined}, and $\U_{t,x}(w)>0$ because $e^{\Lambda_t^x(w)}>0$ and $\varphi(x+B_t)>0$ with positive probability, as $\varphi$ is positive on a non-empty open set and $B_t$ has a positive density. Hence $u(t,x)\in(0,\infty)$ almost surely. \qed

\section{Proof of Proposition~\ref{prop:young}}
\label{sec:young}

Let $\rho\in(0,1)$ and let $w\in\Cr$ be locally $\theta$-Hölder continuous for some $\theta>\frac12$. Choose $\gamma\in(1-\theta,\frac12)$ and $\vartheta\in(1-\gamma,\theta)$, which is possible since $\theta>\frac12$; then $\gamma+\vartheta>1$. Fix $t\ge0$ and $x\in\R$. For every $k\ge1$ we write $J_k=[-k,k]$.

\emph{Step 1: regularisation.} Let $\varrho$ be a smooth probability density supported in $[-1,1]$, $\varrho_n(y):=n\varrho(ny)$, and
\[
w_n:=w*\varrho_n-(w*\varrho_n)(0),
\]
which is a smooth function with $w_n(0)=0$. We show that $w_n\to w$ in $\Cr$ and in $C^\vartheta(J_k)$ for every $k$.

For $|y|\le k$ and $n\ge1$,
\begin{equation}
\label{eq:moll_sup}
|(w*\varrho_n)(y)-w(y)|\le\int_\R|w(y-h)-w(y)|\,\varrho_n(h)\,\dd h\le[w]_{\theta,J_{k+1}}\,n^{-\theta},
\end{equation}
because $\varrho_n$ is supported in $[-1/n,1/n]$. In particular $(w*\varrho_n)(0)\to w(0)=0$, and $w_n\to w$ uniformly on every $J_k$. For the behaviour at infinity, let $s(r):=\sup_{|y'|\ge r}|w(y')|/(1+|y'|)^\rho$, which tends to $0$ as $r\to\infty$ since $w\in\Cr$. For $|y|\ge k\ge2$ and $|h|\le1$, we have $1+|y-h|\le2(1+|y|)$ and $|y-h|\ge k-1$, so that
\[
\frac{|(w*\varrho_n)(y)|}{(1+|y|)^\rho}\le\sup_{|h|\le1}\frac{|w(y-h)|}{(1+|y|)^\rho}\le2^\rho\,s(k-1).
\]
Hence, for $k\ge2$,
\[
\|w_n-w\|_\rho\le\sup_{|y|\le k}|w_n(y)-w(y)|+\bigl(2^\rho+1\bigr)s(k-1)+|(w*\varrho_n)(0)|,
\]
and letting first $n\to\infty$ and then $k\to\infty$ gives $\|w_n-w\|_\rho\to0$.

Let $f_n:=w*\varrho_n-w$. For $u,v\in J_k$, $|f_n(u)-f_n(v)|\le2[w]_{\theta,J_{k+1}}|u-v|^\theta$, since $w*\varrho_n$ has the same Hölder seminorm as $w$ on $J_k$, up to replacing $J_k$ by $J_{k+1}$. Also $|f_n(u)-f_n(v)|\le2\sup_{J_k}|f_n|$. Writing $\mathsf h:=2[w]_{\theta,J_{k+1}}$ and $\mathsf s_n:=2\sup_{J_k}|f_n|$, we obtain
\[
|f_n(u)-f_n(v)|\le\min\bigl(\mathsf h|u-v|^\theta,\mathsf s_n\bigr)\le\bigl(\mathsf h|u-v|^\theta\bigr)^{\vartheta/\theta}\mathsf s_n^{1-\vartheta/\theta}
=\mathsf h^{\vartheta/\theta}\mathsf s_n^{1-\vartheta/\theta}|u-v|^\vartheta,
\]
so that $[w_n-w]_{\vartheta,J_k}=[f_n]_{\vartheta,J_k}\le\mathsf h^{\vartheta/\theta}\mathsf s_n^{1-\vartheta/\theta}\to0$ by~\eqref{eq:moll_sup}.

\emph{Step 2: the smooth case.} Since $w_n$ is smooth, Itô's formula for $F_{w_n}\in C^\infty$ and the occupation density formula~\eqref{eq:occupation} give, almost surely,
\[
\Lambda_t^x(w_n)=\int_0^tw_n'(x+B_s)\,\dd s=\int_\R L^x_t(y)\,w_n'(y)\,\dd y=\int_{J}L^x_t\,\dd w_n,
\]
for every compact interval $J$ containing $[x-B^*_t,x+B^*_t]$, which contains the support of $L^x_t$; the last integral is a Riemann--Stieltjes, hence Young, integral.

\emph{Step 3: passage to the limit.} By Lemma~\ref{lem:U_continuous} and Step~1, $\Lambda_t^x(w_n)\to\Lambda_t^x(w)$ in $L^2(\PP^B)$, hence almost surely along a subsequence. On the other hand, almost surely, $L^x_t$ is $\gamma$-Hölder continuous, and with $k:=\lceil|x|+B^*_t\rceil+1$ and $J=J_k$, Step~1 and the continuity of the Young integral on $C^\gamma(J)\times C^\vartheta(J)$ give
\[
\int_JL^x_t\,\dd w_n\longrightarrow\int_JL^x_t\,\dd w\qquad\text{almost surely.}
\]
Comparing the two limits along the subsequence, $\Lambda_t^x(w)=\int_JL^x_t\,\dd w$ almost surely. The integral does not depend on the choice of the interval $J$ containing the support of $L^x_t$, as one sees on Riemann sums; we denote it by $\int_\R L^x_t\,\dd w$.

\emph{Step 4: the limit field.} Let $\rho\in(H,1)$. By Proposition~\ref{prop:Z} and the proof of Theorem~\ref{thm:W}(ii), almost every path of $\beta Z$ belongs to $\Cr$ and is locally Hölder continuous of some order $\theta\in(\frac12,H)$. For every such path, Step~3 applies with $w=\beta Z$, and $\int_\R L^x_t\,\dd(\beta Z)=\beta\int_\R L^x_t\,\dd Z$ by linearity. Taking $\E^B$ gives the representation of $u(t,x)$. \qed

\section{Moments of the limit}
\label{sec:mean}

This section contains the proofs of Theorem~\ref{thm:mean} and Proposition~\ref{prop:moments}. In Sections~\ref{subsec:rosenblatt_integrals} and~\ref{subsec:proof_mean}, $m=2$, so that $H=1-\alpha$, $\psi(\eta)=|\eta|^{(\alpha-1)/2}$, $\sigma:=\sigma_2>0$, and $Z(y)=\sigma^{-1}I_2(K_0^y)$ is the Rosenblatt process. We fix $\gamma\in(1-H,\frac12)$ and $\theta\in(1-\gamma,H)$.

\subsection{Rosenblatt integrals of deterministic functions}
\label{subsec:rosenblatt_integrals}

\begin{lemma}
\label{lem:X_ell}
Let $\ell\ge0$ be a $\gamma$-Hölder continuous function with support in a compact interval $J$, and let $X_\ell:=\int_J\ell\,\dd Z$, the Young integral being defined pathwise. Then $X_\ell\in\mathcal H_2$, so that $X_\ell=I_2(K_\ell)$ for a real symmetric $K_\ell\in L^2(\R^2)$. Moreover, if $\chi\in L^2(\R)$ has norm one and $y\mapsto q_\chi(y):=\E[Z(y)H_2(I_1(\chi))]$ is of class $C^1$, then
\begin{equation}
\label{eq:X_ell_proj}
\E\bigl[X_\ell\,H_2(I_1(\chi))\bigr]=\int_J\ell(y)\,q_\chi'(y)\,\dd y .
\end{equation}
\end{lemma}

\begin{proof}
Let $\pi_N$ be the dyadic partition of $J$ into $2^N$ intervals of equal length, and let $S_N:=\sum_k\ell(y_k)(Z(y_{k+1})-Z(y_k))$ be the corresponding Riemann sum, which belongs to $\mathcal H_2$ by Proposition~\ref{prop:Z}. By~\eqref{eq:young_loeve}, almost surely,
\begin{align*}
|X_\ell-S_N|&\le c_{\gamma,\theta}[\ell]_{\gamma,J}[Z]_{\theta,J}\,2^N\Bigl(\frac{|J|}{2^N}\Bigr)^{\gamma+\theta}\\
&=c_{\gamma,\theta}[\ell]_{\gamma,J}[Z]_{\theta,J}\,|J|^{\gamma+\theta}\,2^{-N(\gamma+\theta-1)},
\end{align*}
which tends to $0$, and whose square is integrable by Proposition~\ref{prop:Z}. Hence $S_N\to X_\ell$ in $L^2(\Omega)$, and $X_\ell$ belongs to the closed subspace $\mathcal H_2$. Every element of $\mathcal H_2$ is of the form $I_2(K)$ with $K\in L^2(\R^2)$ real and symmetric, so $X_\ell=I_2(K_\ell)$ for such a kernel $K_\ell$. Finally,
\[
\E\bigl[S_NH_2(I_1(\chi))\bigr]=\sum_k\ell(y_k)\bigl(q_\chi(y_{k+1})-q_\chi(y_k)\bigr)\longrightarrow\int_J\ell(y)\,q_\chi'(y)\,\dd y,
\]
since $\ell q_\chi'$ is continuous on $J$, and the left-hand side tends to $\E[X_\ell H_2(I_1(\chi))]$ by $L^2$ convergence.
\end{proof}

\begin{lemma}
\label{lem:test_functions}
Let $\Gamma:\R\to\R$ be a real Schwartz function, $\Gamma\not\equiv0$, and define $\chi_\Gamma\in L^2(\R)$ by its Fourier transform
\[
\widehat{\chi_\Gamma}(\eta):=\frac{\hat\Gamma(\eta)}{2\pi\,\psi(\eta)}=\frac{|\eta|^{(1-\alpha)/2}\,\hat\Gamma(\eta)}{2\pi} .
\]
Then $\chi_\Gamma$ is real, $\|\chi_\Gamma\|^2_{L^2(\R)}=(2\pi)^{-3}N(\Gamma)$ with $N(\Gamma):=\int_\R|\hat\Gamma(\eta)|^2|\eta|^{1-\alpha}\,\dd\eta\in(0,\infty)$, and, with $\tilde\chi_\Gamma:=\chi_\Gamma/\|\chi_\Gamma\|$,
\begin{equation}
\label{eq:q_chi}
q_{\tilde\chi_\Gamma}(y)=\E\bigl[Z(y)H_2(I_1(\tilde\chi_\Gamma))\bigr]=\frac{8\pi^2\,b(0)}{\sigma\,N(\Gamma)}\int_0^y\Gamma(u)^2\,\dd u,\qquad y\in\R .
\end{equation}
Moreover, the set $\{\chi_\Gamma:\Gamma\text{ real Schwartz}\}$ is a dense linear subspace of the real space $L^2(\R)$.
\end{lemma}

\begin{proof}
Since $\Gamma$ is real, $\hat\Gamma(-\eta)=\overline{\hat\Gamma(\eta)}$, and since $\psi$ is even, the same holds for $\widehat{\chi_\Gamma}$, so $\chi_\Gamma$ is real. Since $\hat\Gamma$ is a Schwartz function, $|\widehat{\chi_\Gamma}|^2=(2\pi)^{-2}|\hat\Gamma|^2|\eta|^{1-\alpha}$ is integrable, and $\|\chi_\Gamma\|^2=(2\pi)^{-1}\|\widehat{\chi_\Gamma}\|^2=(2\pi)^{-3}N(\Gamma)$; $N(\Gamma)>0$ because $\hat\Gamma\not\equiv0$.

By~\eqref{eq:Hq_Iq}, \eqref{eq:isometry} and Parseval's identity on $\R^2$,
\begin{align*}
\E\bigl[Z(y)\,I_2(\chi_\Gamma\otimes\chi_\Gamma)\bigr]
&=\frac2\sigma\,\bigl\langle K_0^y,\chi_\Gamma\otimes\chi_\Gamma\bigr\rangle_{L^2(\R^2)}
=\frac2{\sigma(2\pi)^2}\int_{\R^2}\widehat{K_0^y}(\eta_1,\eta_2)\,\overline{\widehat{\chi_\Gamma}(\eta_1)\widehat{\chi_\Gamma}(\eta_2)}\,\dd\eta\\
&=\frac{2\cdot2\pi b(0)}{\sigma(2\pi)^2}\int_{\R^2}e_y(\eta_1+\eta_2)\,\psi(\eta_1)\overline{\widehat{\chi_\Gamma}(\eta_1)}\,\psi(\eta_2)\overline{\widehat{\chi_\Gamma}(\eta_2)}\,\dd\eta\\
&=\frac{b(0)}{\pi\sigma}\cdot\frac1{(2\pi)^2}\int_{\R^2}\int_0^ye^{-iu(\eta_1+\eta_2)}\,\dd u\;\overline{\hat\Gamma(\eta_1)}\,\overline{\hat\Gamma(\eta_2)}\,\dd\eta .
\end{align*}
The integrand is bounded by $|y|\,|\hat\Gamma(\eta_1)||\hat\Gamma(\eta_2)|$, which is integrable, so Fubini's theorem applies, and by Fourier inversion,
\[
\frac1{2\pi}\int_\R e^{-iu\eta}\,\overline{\hat\Gamma(\eta)}\,\dd\eta=\overline{\frac1{2\pi}\int_\R e^{iu\eta}\hat\Gamma(\eta)\,\dd\eta}=\overline{\Gamma(u)}=\Gamma(u).
\]
Hence
\[
\E\bigl[Z(y)\,I_2(\chi_\Gamma\otimes\chi_\Gamma)\bigr]=\frac{b(0)}{\pi\sigma}\int_0^y\Gamma(u)^2\,\dd u .
\]
Since $H_2(I_1(\tilde\chi_\Gamma))=I_2(\chi_\Gamma\otimes\chi_\Gamma)/\|\chi_\Gamma\|^2$ and $\|\chi_\Gamma\|^2=(2\pi)^{-3}N(\Gamma)$, we obtain~\eqref{eq:q_chi}, because $(2\pi)^3/\pi=8\pi^2$.

The map $\Gamma\mapsto\chi_\Gamma$ is linear. Let $s$ be a smooth function with compact support in $\R\setminus\{0\}$ such that $s(-\eta)=\overline{s(\eta)}$. Then $\hat\Gamma:=2\pi\psi s$ is smooth with compact support and satisfies the same symmetry, so it is the Fourier transform of a real Schwartz function $\Gamma$, and $\widehat{\chi_\Gamma}=s$. Such functions $s$ are dense in $\{\hat\chi:\chi\in L^2(\R)\text{ real}\}$, and the density claim follows from Plancherel's theorem.
\end{proof}

\begin{lemma}
\label{lem:top_eigenvalue}
Let $\ell$ and $J$ be as in Lemma~\ref{lem:X_ell}, and let $\mathcal A_\ell$ be the Hilbert--Schmidt operator with kernel $K_\ell$. Then $\mathcal A_\ell$ is non-negative, and for every real Schwartz function $\Gamma\not\equiv0$, its largest eigenvalue satisfies
\[
\lambda_{\max}(\mathcal A_\ell)\ge\frac{4\pi^2\,b(0)}{\sigma\,N(\Gamma)}\int_J\ell(y)\,\Gamma(y)^2\,\dd y .
\]
\end{lemma}

\begin{proof}
By Lemma~\ref{lem:second_chaos}, \eqref{eq:X_ell_proj} and~\eqref{eq:q_chi},
\[
\bigl\langle\mathcal A_\ell\tilde\chi_\Gamma,\tilde\chi_\Gamma\bigr\rangle=\frac12\E\bigl[X_\ell H_2(I_1(\tilde\chi_\Gamma))\bigr]
=\frac12\cdot\frac{8\pi^2b(0)}{\sigma N(\Gamma)}\int_J\ell(y)\Gamma(y)^2\,\dd y\ge0 .
\]
Since the quadratic form of $\mathcal A_\ell$ is homogeneous of degree two, it is non-negative on the dense subspace of Lemma~\ref{lem:test_functions}, hence everywhere by continuity, so $\mathcal A_\ell$ is non-negative. For a compact non-negative self-adjoint operator, $\lambda_{\max}(\mathcal A_\ell)=\sup_{\|\chi\|=1}\langle\mathcal A_\ell\chi,\chi\rangle$, and the lower bound follows.
\end{proof}

\subsection{Proof of Theorem~\ref{thm:mean}}
\label{subsec:proof_mean}

Let $m=2$, $c_2>0$, so that $\beta>0$, and let $t>0$, $x\in\R$, $\varphi\ge0$ bounded and continuous with $\varphi(x)>0$. The almost sure finiteness of $u(t,x)$ was proved in Theorem~\ref{thm:u}.

\emph{Step 1: a measurable version of the exponent.} For $k\ge1$, let $\Xi_k$ be the limit, when it exists, of the Riemann sums of $\int_{J_k}L^x_t\,\dd Z$ along the dyadic partitions of $J_k$, and $\Xi_k:=0$ otherwise. With $k_B:=\lceil|x|+B^*_t\rceil+1$, the random variable $\Xi:=\Xi_{k_B}$ is measurable with respect to the product $\sigma$-field, and, by the Young--Loève inequality, $\Xi=\int_\R L^x_t\,\dd Z$ whenever $L^x_t$ is $\gamma$-Hölder and $Z$ is $\theta$-Hölder on $J_{k_B}$, which holds almost surely. By Proposition~\ref{prop:young}, for almost every realisation of $Z$, $u(t,x)=\E^B[\varphi(x+B_t)e^{\beta\Xi}]$. By Tonelli's theorem,
\begin{equation}
\label{eq:mean_tonelli}
\E\bigl[u(t,x)\bigr]=\E^B\Bigl[\varphi(x+B_t)\,\E\bigl[e^{\beta\Xi}\bigm|B\bigr]\Bigr],
\end{equation}
where, for a fixed Brownian path, $\E[e^{\beta\Xi}\mid B]$ denotes the integral with respect to $Z$ only.

\emph{Step 2: choice of the test function.} Let $\Gamma_1(u):=e^{1-u^2}$, a real Schwartz function with $\Gamma_1\ge1$ on $[-1,1]$, and for $r>0$ let $\Gamma_r(u):=\Gamma_1((u-x)/r)$. Then $\hat\Gamma_r(\eta)=r\,e^{-ix\eta}\hat\Gamma_1(r\eta)$, and the substitution $\zeta=r\eta$ gives
\[
N(\Gamma_r)=\int_\R r^2|\hat\Gamma_1(r\eta)|^2|\eta|^{1-\alpha}\,\dd\eta=r^{2-1-(1-\alpha)}\int_\R|\hat\Gamma_1(\zeta)|^2|\zeta|^{1-\alpha}\,\dd\zeta=r^\alpha N(\Gamma_1).
\]
By continuity of $\varphi$, choose $r>0$ so small that $\varphi\ge\varphi(x)/2$ on $[x-r,x+r]$ and
\begin{equation}
\label{eq:r_choice}
\beta\,\frac{4\pi^2\,b(0)\,t}{\sigma\,N(\Gamma_1)\,r^\alpha}>\frac12 .
\end{equation}

\emph{Step 3: the conditional exponential moment.} Let $A_r:=\{B^*_t<r\}$, which has positive probability by Lemma~\ref{lem:confinement}. Fix a Brownian path in $A_r$ for which $\ell:=L^x_t$ is $\gamma$-Hölder continuous. Then $\ell\ge0$, $\int\ell=t$, the support of $\ell$ is contained in $[x-r,x+r]$, and $\Gamma_r\ge1$ there, so that $\int\ell\,\Gamma_r^2\ge t$. By Lemma~\ref{lem:top_eigenvalue} and~\eqref{eq:r_choice}, the operator $\beta\mathcal A_\ell$ has an eigenvalue $\lambda_1>\frac12$. By Lemma~\ref{lem:second_chaos}, $\beta X_\ell=\lambda_1(\xi_1^2-1)+X'$, where $\xi_1$ is a standard Gaussian variable and $X'$, the sum of the remaining terms, is independent of $\xi_1$. Hence, by independence and Tonelli's theorem,
\[
\E\bigl[e^{\beta X_\ell}\bigr]=\E\bigl[e^{\lambda_1(\xi_1^2-1)}\bigr]\,\E\bigl[e^{X'}\bigr]=+\infty,
\]
because $\E[e^{\lambda_1\xi_1^2}]=+\infty$ for $\lambda_1\ge\frac12$, and $\E[e^{X'}]>0$. Since $\Xi=X_\ell$ for this path, $\E[e^{\beta\Xi}\mid B]=+\infty$.

\emph{Step 4: conclusion.} On $A_r$, $|B_t|<r$, so $\varphi(x+B_t)\ge\varphi(x)/2>0$. By~\eqref{eq:mean_tonelli} and Step~3,
\[
\E\bigl[u(t,x)\bigr]\ge\E^B\Bigl[\mathbf 1_{A_r}\,\frac{\varphi(x)}2\cdot(+\infty)\Bigr]=+\infty,
\]
since $\PP^B(A_r)>0$. \qed

\subsection{Exponential moments of two Brownian functionals}
\label{subsec:brownian_exp}

Let $L_t(y)$ be the local time of $B$ at level $y$ and time $t$, in its jointly continuous version, and let
\[
L^*_t:=\sup_{y\in\R}L_t(y),\qquad t\ge0 .
\]
For every $x\in\R$, the local time of $x+B$ is $L^x_t(y)=L_t(y-x)$, so that $\sup_yL^x_t(y)=L^*_t$. We use the filtration of $B$, completed by the $\PP^B$-null sets.

\begin{lemma}[Maximum of the local time]
\label{lem:sup_local_time}
For every $t\ge0$ and every $c\ge0$, $\E^B\bigl[e^{cL^*_t}\bigr]<\infty$.
\end{lemma}

\begin{proof}
\emph{Step 1: path properties.} We show that, almost surely, $s\mapsto L^*_s$ is continuous and non-decreasing on $[0,\infty)$, with $L^*_0=0$ and $L^*_s\to\infty$ as $s\to\infty$. Fix $T>0$. For $s\le T$, $L_s(y)=0$ when $|y|>B^*_T$, so that $L^*_s=\sup_{|y|\le B^*_T}L_s(y)$. The map $(s,y)\mapsto L_s(y)$ is uniformly continuous on the compact set $[0,T]\times[-B^*_T,B^*_T]$, hence $s\mapsto L^*_s$ is continuous on $[0,T]$. It is non-decreasing because each $s\mapsto L_s(y)$ is, and $L^*_0=0$ because $L_0\equiv0$. Finally, $L^*_s\ge L_s(0)$, and by Lévy's theorem~\cite[Chapter~VI, Theorem~2.3]{RevuzYor1999}, $L_s(0)$ has the law of $\sup_{r\le s}B_r$, hence of $\sqrt s\,|N|$ with $N$ standard Gaussian. For every $K>0$,
\[
\PP^B\Bigl(\lim_{r\to\infty}L_r(0)\le K\Bigr)\le\PP^B\bigl(L_s(0)\le K\bigr)=\PP\bigl(|N|\le K/\sqrt s\bigr)\xrightarrow[s\to\infty]{}0,
\]
where the limit exists because $r\mapsto L_r(0)$ is non-decreasing. Hence $L_s(0)\to\infty$, and $L^*_s\to\infty$, almost surely.

\emph{Step 2: additivity.} Let $\sigma$ be an almost surely finite stopping time, let $\tilde B_s:=B_{\sigma+s}-B_\sigma$, which is a Brownian motion independent of $\F^B_\sigma$ by the strong Markov property, and let $\tilde L_s(z)$ be its jointly continuous local time. We claim that, almost surely,
\begin{equation}
\label{eq:lt_additivity}
L_{\sigma+s}(y)=L_\sigma(y)+\tilde L_s(y-B_\sigma)\qquad\text{for all }s\ge0,\ y\in\R .
\end{equation}
Indeed, by the occupation density formula~\eqref{eq:occupation} for $B$ and for $\tilde B$, almost surely, for every $s\ge0$ and every continuous $h\ge0$ with compact support,
\begin{align*}
\int_\R h(y)\bigl(L_{\sigma+s}(y)-L_\sigma(y)\bigr)\dd y
&=\int_\sigma^{\sigma+s}h(B_r)\,\dd r
=\int_0^sh\bigl(B_\sigma+\tilde B_r\bigr)\dd r\\
&=\int_\R h(B_\sigma+z)\,\tilde L_s(z)\,\dd z
=\int_\R h(y)\,\tilde L_s(y-B_\sigma)\,\dd y .
\end{align*}
Both $y\mapsto L_{\sigma+s}(y)-L_\sigma(y)$ and $y\mapsto\tilde L_s(y-B_\sigma)$ are continuous, and two continuous functions with the same integral against every such $h$ coincide. This proves~\eqref{eq:lt_additivity}, which implies, with $\tilde L^*_s:=\sup_z\tilde L_s(z)$,
\begin{equation}
\label{eq:sup_additivity}
L^*_{\sigma+s}\le L^*_\sigma+\tilde L^*_s,\qquad s\ge0 .
\end{equation}

\emph{Step 3: a renewal argument.} For $k\ge0$, let $\sigma_k:=\inf\{s\ge0:L^*_s\ge k\}$. By Step~1, $L^*$ is continuous and adapted, so $\sigma_k$ is a stopping time; it is finite almost surely because $L^*_s\to\infty$, and $L^*_{\sigma_k}=k$ by continuity. Let $B^{(k)}_s:=B_{\sigma_k+s}-B_{\sigma_k}$, let $L^{(k)*}_s$ be the maximum of its local time at time $s$, and
\[
\rho_k:=\inf\bigl\{s\ge0:L^{(k)*}_s\ge1\bigr\}.
\]
For $s<\rho_k$, \eqref{eq:sup_additivity} with $\sigma=\sigma_k$ gives $L^*_{\sigma_k+s}\le k+L^{(k)*}_s<k+1$. Hence $\sigma_{k+1}\ge\sigma_k+\rho_k$, and by induction, since $\sigma_0=0$,
\begin{equation}
\label{eq:sigma_k_lower}
\sigma_k\ge\rho_0+\rho_1+\dots+\rho_{k-1},\qquad k\ge1 .
\end{equation}
By the strong Markov property, $B^{(k)}$ is a Brownian motion independent of $\F^B_{\sigma_k}$, so $\rho_k$ has the law of $\rho_0$ and is independent of $\F^B_{\sigma_k}$. For $j<k$, $\rho_j$ is $\F^B_{\sigma_k}$-measurable. Indeed, by~\eqref{eq:lt_additivity}, $L^{(j)*}_{r-\sigma_j}$ is, for $r\ge\sigma_j$, a continuous function of the path of $B$ up to time $r$, so $\tau_j:=\sigma_j+\rho_j$ is a stopping time; then $\rho_j=\tau_j-\sigma_j$ is $\F^B_{\tau_j}$-measurable, and $\tau_j\le\sigma_{j+1}\le\sigma_k$. For $\lambda>0$, let $q_\lambda:=\E^B[e^{-\lambda\rho_0}]$. Conditioning on $\F^B_{\sigma_{k-1}}$ and using induction on $k$,
\begin{align*}
\E^B\Bigl[e^{-\lambda(\rho_0+\dots+\rho_{k-1})}\Bigr]
&=\E^B\Bigl[e^{-\lambda(\rho_0+\dots+\rho_{k-2})}\,\E^B\bigl[e^{-\lambda\rho_{k-1}}\bigm|\F^B_{\sigma_{k-1}}\bigr]\Bigr]\\
&=q_\lambda\,\E^B\Bigl[e^{-\lambda(\rho_0+\dots+\rho_{k-2})}\Bigr]
=\dots=q_\lambda^k .
\end{align*}

\emph{Step 4: conclusion.} Since $L^*$ is continuous with $L^*_0=0$, $\rho_0>0$ almost surely, so $q_\lambda\to0$ as $\lambda\to\infty$ by dominated convergence. By~\eqref{eq:sigma_k_lower} and Markov's inequality, for $k\ge1$ and $\lambda>0$,
\begin{align*}
\PP^B\bigl(L^*_t\ge k\bigr)&=\PP^B(\sigma_k\le t)
\le\PP^B\bigl(\rho_0+\dots+\rho_{k-1}\le t\bigr)\\
&=\PP^B\Bigl(e^{-\lambda(\rho_0+\dots+\rho_{k-1})}\ge e^{-\lambda t}\Bigr)
\le e^{\lambda t}q_\lambda^k .
\end{align*}
Given $c\ge0$, choose $\lambda>0$ such that $e^cq_\lambda\le\frac12$. Then
\begin{align*}
\E^B\bigl[e^{cL^*_t}\bigr]
&\le\sum_{k\ge0}e^{c(k+1)}\,\PP^B\bigl(k\le L^*_t<k+1\bigr)
\le e^c\sum_{k\ge0}e^{ck}\,\PP^B\bigl(L^*_t\ge k\bigr)\\
&\le e^{c}\Bigl(1+e^{\lambda t}\sum_{k\ge1}\bigl(e^cq_\lambda\bigr)^k\Bigr)
\le e^{c}\bigl(1+e^{\lambda t}\bigr)<\infty .
\end{align*}
\end{proof}

\begin{lemma}[A double integral]
\label{lem:double_integral}
Let $a\in(0,1)$, $t\ge0$ and $x\in\R$, and let $I^{(a)}_t:=\int_0^t\int_0^t|B_s-B_r|^{-a}\,\dd s\,\dd r$. Then, almost surely,
\begin{equation}
\label{eq:double_integral}
I^{(a)}_t=\int_\R\int_\R L^x_t(y)\,L^x_t(y')\,|y-y'|^{-a}\,\dd y\,\dd y'\le t\Bigl(\frac{2L^*_t}{1-a}+t\Bigr),
\end{equation}
and $\E^B[e^{\lambda I^{(a)}_t}]<\infty$ for every $\lambda\ge0$.
\end{lemma}

\begin{proof}
Since $|B_s-B_r|=|(x+B_s)-(x+B_r)|$, we may compute with $x+B$. Fix $r\in[0,t]$. By~\eqref{eq:occupation} applied to the non-negative function $h_r(z):=|z-(x+B_r)|^{-a}$, almost surely, simultaneously for all $r$,
\[
\int_0^t|B_s-B_r|^{-a}\,\dd s=\int_\R|y-(x+B_r)|^{-a}\,L^x_t(y)\,\dd y=g(x+B_r),
\qquad g(z):=\int_\R|y-z|^{-a}\,L^x_t(y)\,\dd y .
\]
The function $g$ is non-negative and measurable, and a second application of~\eqref{eq:occupation} gives
\[
I^{(a)}_t=\int_0^tg(x+B_r)\,\dd r=\int_\R g(y')\,L^x_t(y')\,\dd y'
=\int_\R\int_\R L^x_t(y)\,L^x_t(y')\,|y-y'|^{-a}\,\dd y\,\dd y' .
\]
For every $z\in\R$, splitting the integral defining $g$ according to $|y-z|\le1$ or $|y-z|>1$, and using $L^x_t\le L^*_t$ and $\int L^x_t=t$,
\begin{align*}
g(z)&=\int_{|y-z|\le1}|y-z|^{-a}\,L^x_t(y)\,\dd y+\int_{|y-z|>1}|y-z|^{-a}\,L^x_t(y)\,\dd y\\
&\le L^*_t\int_{-1}^1|v|^{-a}\,\dd v+\int_\R L^x_t(y)\,\dd y
=\frac{2L^*_t}{1-a}+t .
\end{align*}
Hence $I^{(a)}_t\le\bigl(\frac{2L^*_t}{1-a}+t\bigr)\int_\R L^x_t(y')\,\dd y'=t\bigl(\frac{2L^*_t}{1-a}+t\bigr)$, which is~\eqref{eq:double_integral}. Finally, for $\lambda\ge0$,
\[
\E^B\bigl[e^{\lambda I^{(a)}_t}\bigr]\le e^{\lambda t^2}\,\E^B\Bigl[\exp\Bigl(\frac{2\lambda t}{1-a}\,L^*_t\Bigr)\Bigr]<\infty
\]
by Lemma~\ref{lem:sup_local_time}.
\end{proof}

\subsection{Proof of Proposition~\ref{prop:moments}}
\label{subsec:proof_moments}

In this subsection $m\in\{1,2\}$ and $H=1-m\alpha/2$, and we fix $\gamma\in(1-H,\frac12)$ and $\theta\in(1-\gamma,H)$ for this value of $H$.

\begin{lemma}[Variance of Hermite integrals]
\label{lem:var_general}
Let $\ell\ge0$ be a $\gamma$-Hölder continuous function with support in a compact interval $J$, and let $X_\ell:=\int_J\ell\,\dd Z$, the Young integral being defined pathwise. Then $X_\ell$ belongs to the $m$-th Wiener chaos $\mathcal H_m$, $\E[X_\ell]=0$, and
\begin{equation}
\label{eq:var_general}
\Var(X_\ell)=H(2H-1)\int_J\int_J\ell(y)\,\ell(y')\,|y-y'|^{2H-2}\,\dd y\,\dd y' .
\end{equation}
If $m=1$, $X_\ell$ is a centred Gaussian variable.
\end{lemma}

\begin{proof}
Let $S_N:=\sum_k\ell(y_k)(Z(y_{k+1})-Z(y_k))$ be the Riemann sums along the dyadic partitions $\pi_N$ of $J$. The argument of the proof of Lemma~\ref{lem:X_ell}, which only uses Proposition~\ref{prop:Z} and~\eqref{eq:young_loeve}, and therefore holds for every $m$, shows that $S_N\in\mathcal H_m$ and $S_N\to X_\ell$ in $L^2(\Omega)$. Since $\mathcal H_m$ is closed and $\E[S_N]=0$, $X_\ell\in\mathcal H_m$ and $\E[X_\ell]=0$.

We compute $\Var(S_N)$. By Proposition~\ref{prop:Z}, $\E[Z(y)Z(y')]=\mathcal R(y,y'):=\frac12(|y|^{2H}+|y'|^{2H}-|y-y'|^{2H})$. The function $z\mapsto|z|^{2H-1}\operatorname{sgn}(z)$ is absolutely continuous, with derivative $(2H-1)|z|^{2H-2}$, which is locally integrable since $2H-2>-1$. Hence $\partial_y\mathcal R(y,y')=H\bigl(|y|^{2H-1}\operatorname{sgn}(y)-|y-y'|^{2H-1}\operatorname{sgn}(y-y')\bigr)$, and $y'\mapsto\partial_y\mathcal R(y,y')$ is absolutely continuous with derivative $H(2H-1)|y-y'|^{2H-2}$. Therefore, for intervals $[u,v]$ and $[u',v']$,
\begin{align*}
\E\bigl[(Z(v)-Z(u))(Z(v')-Z(u'))\bigr]
&=\mathcal R(v,v')-\mathcal R(u,v')-\mathcal R(v,u')+\mathcal R(u,u')\\
&=\int_u^v\bigl(\partial_y\mathcal R(y,v')-\partial_y\mathcal R(y,u')\bigr)\dd y\\
&=H(2H-1)\int_u^v\int_{u'}^{v'}|y-y'|^{2H-2}\,\dd y'\,\dd y .
\end{align*}
Let $\ell_N:=\sum_k\ell(y_k)\mathbf 1_{[y_k,y_{k+1})}$. Summing over the intervals of $\pi_N$,
\begin{align*}
\Var(S_N)&=\sum_{k,j}\ell(y_k)\,\ell(y_j)\,\E\bigl[(Z(y_{k+1})-Z(y_k))(Z(y_{j+1})-Z(y_j))\bigr]\\
&=H(2H-1)\sum_{k,j}\ell(y_k)\,\ell(y_j)\int_{y_k}^{y_{k+1}}\int_{y_j}^{y_{j+1}}|y-y'|^{2H-2}\,\dd y'\,\dd y\\
&=H(2H-1)\int_J\int_J\ell_N(y)\,\ell_N(y')\,|y-y'|^{2H-2}\,\dd y\,\dd y' .
\end{align*}
Since $\ell$ is continuous on $J$, $\ell_N\to\ell$ uniformly on $J$, and $0\le\ell_N\le\sup_J\ell$. The function $|y-y'|^{2H-2}$ is integrable on $J\times J$, since
\[
\int_J\int_J|y-y'|^{2H-2}\,\dd y\,\dd y'\le|J|\int_{-|J|}^{|J|}|v|^{2H-2}\,\dd v=\frac{2|J|^{2H}}{2H-1}<\infty .
\]
By dominated convergence, $\Var(S_N)$ converges to the right-hand side of~\eqref{eq:var_general}, and $\Var(S_N)\to\Var(X_\ell)$ by $L^2$ convergence. If $m=1$, each $S_N$ belongs to the first chaos, hence is a centred Gaussian variable, and so is its $L^2$ limit $X_\ell$, since $\E[e^{i\theta' S_N}]=e^{-\theta'^2\Var(S_N)/2}$ converges to $\E[e^{i\theta' X_\ell}]$ for every $\theta'\in\R$.
\end{proof}

\begin{lemma}[Second-chaos variables with non-positive coefficients]
\label{lem:cond_exp}
Let $(\xi_k)_{k\ge1}$ be independent standard Gaussian variables, and let $(\lambda_k)$ be real numbers with $\lambda_k\le0$ and $\sum_k\lambda_k^2<\infty$. Then $X:=\sum_k\lambda_k(\xi_k^2-1)$ converges in $L^2$, $\Var(X)=2\sum_k\lambda_k^2$, and
\[
\E\bigl[e^X\bigr]\le\exp\Bigl(\sum_k\lambda_k^2\Bigr)=\exp\Bigl(\frac{\Var(X)}2\Bigr).
\]
\end{lemma}

\begin{proof}
The terms $\lambda_k(\xi_k^2-1)$ are independent and centred, with variance $2\lambda_k^2$, so the series converges in $L^2$ and $\Var(X)=2\sum_k\lambda_k^2$. For $\lambda\le0$, $\E[e^{\lambda\xi^2}]=(1-2\lambda)^{-1/2}$ for a standard Gaussian $\xi$, so $\E[e^{\lambda(\xi^2-1)}]=e^{-\lambda}(1-2\lambda)^{-1/2}$. The function $f(\lambda):=\lambda^2+\lambda+\frac12\log(1-2\lambda)$ satisfies $f(0)=0$ and
\[
f'(\lambda)=2\lambda+1-\frac1{1-2\lambda}=\frac{(1+2\lambda)(1-2\lambda)-1}{1-2\lambda}=-\frac{4\lambda^2}{1-2\lambda}\le0,\qquad\lambda\le0,
\]
so $f\ge0$ on $(-\infty,0]$, that is, $-\lambda-\frac12\log(1-2\lambda)\le\lambda^2$. Let $X_N:=\sum_{k\le N}\lambda_k(\xi_k^2-1)$. By independence,
\[
\E\bigl[e^{X_N}\bigr]=\prod_{k\le N}\frac{e^{-\lambda_k}}{\sqrt{1-2\lambda_k}}
=\exp\Bigl(\sum_{k\le N}\Bigl(-\lambda_k-\frac12\log(1-2\lambda_k)\Bigr)\Bigr)
\le\exp\Bigl(\sum_k\lambda_k^2\Bigr).
\]
Since $X_N\to X$ in $L^2$, hence almost surely along a subsequence, Fatou's lemma gives the claim.
\end{proof}

\begin{proof}[Proof of Proposition~\ref{prop:moments}]
For $t=0$, $u(0,x)=\varphi(x)$ and there is nothing to prove; let $t>0$.

\emph{Step 1: reduction to a conditional exponential moment.} Let $\Xi$ be the random variable constructed in Step~1 of the proof of Theorem~\ref{thm:mean}, with $\gamma$ and $\theta$ as above; that construction uses only Proposition~\ref{prop:young} and the Young--Loève inequality, and holds for every $m$. For almost every realisation of $Z$, $u(t,x)=\E^B[\varphi(x+B_t)e^{\beta\Xi}]$, hence, by Jensen's inequality for the convex function $z\mapsto z^p$ on $[0,\infty)$,
\[
|u(t,x)|^p\le\|\varphi\|_\infty^p\Bigl(\E^B\bigl[e^{\beta\Xi}\bigr]\Bigr)^p\le\|\varphi\|_\infty^p\,\E^B\bigl[e^{p\beta\Xi}\bigr].
\]
By Tonelli's theorem,
\begin{equation}
\label{eq:moment_tonelli}
\E\bigl[|u(t,x)|^p\bigr]\le\|\varphi\|_\infty^p\,\E^B\Bigl[\E\bigl[e^{p\beta\Xi}\bigm|B\bigr]\Bigr].
\end{equation}

\emph{Step 2: the conditional exponential moment.} Fix a Brownian path for which $\ell:=L^x_t$ is $\gamma$-Hölder continuous and~\eqref{eq:double_integral} holds with $a:=2-2H=m\alpha\in(0,1)$; this holds almost surely. Then $\Xi=X_\ell$, with $J$ a compact interval containing the support of $\ell$, and by Lemma~\ref{lem:var_general} and~\eqref{eq:double_integral},
\begin{equation}
\label{eq:cond_var}
\Var(X_\ell)=H(2H-1)\int_\R\int_\R L^x_t(y)\,L^x_t(y')\,|y-y'|^{2H-2}\,\dd y\,\dd y'=H(2H-1)\,I^{(2-2H)}_t .
\end{equation}
If $m=1$, $X_\ell$ is centred Gaussian by Lemma~\ref{lem:var_general}, so
\[
\E\bigl[e^{p\beta X_\ell}\bigr]=\exp\Bigl(\frac{p^2\beta^2}2\Var(X_\ell)\Bigr)=\exp\Bigl(\frac{p^2\beta^2H(2H-1)}2\,I^{(2-2H)}_t\Bigr).
\]
If $m=2$ and $c_2<0$, then $\beta<0$ by~\eqref{eq:H_beta}. By Lemma~\ref{lem:X_ell}, $X_\ell=I_2(K_\ell)$, and by Lemma~\ref{lem:second_chaos}, $X_\ell=\sum_k\mu_k(\xi_k^2-1)$, where $(\mu_k)$ are the non-zero eigenvalues of $\mathcal A_\ell$, which are non-negative by Lemma~\ref{lem:top_eigenvalue}. Hence $p\beta X_\ell=\sum_k\lambda_k(\xi_k^2-1)$ with $\lambda_k:=p\beta\mu_k\le0$, and Lemma~\ref{lem:cond_exp} and~\eqref{eq:cond_var} give
\[
\E\bigl[e^{p\beta X_\ell}\bigr]\le\exp\Bigl(\frac{\Var(p\beta X_\ell)}2\Bigr)=\exp\Bigl(\frac{p^2\beta^2H(2H-1)}2\,I^{(2-2H)}_t\Bigr).
\]

\emph{Step 3: conclusion.} By Steps~1 and~2,
\[
\E\bigl[|u(t,x)|^p\bigr]\le\|\varphi\|_\infty^p\,\E^B\Bigl[\exp\Bigl(\frac{p^2\beta^2H(2H-1)}2\,I^{(2-2H)}_t\Bigr)\Bigr],
\]
which is~\eqref{eq:moment_bound}, since $|B_s-B_r|^{2H-2}=|B_s-B_r|^{-(2-2H)}$, and the right-hand side is finite by Lemma~\ref{lem:double_integral} with $a=2-2H\in(0,1)$ and $\lambda=p^2\beta^2H(2H-1)/2$. If $m=1$, $p=1$ and $\varphi\ge0$, then $u(t,x)=\E^B[\varphi(x+B_t)e^{\beta\Xi}]$ without absolute values, Tonelli's theorem gives $\E[u(t,x)]=\E^B[\varphi(x+B_t)\E[e^{\beta\Xi}\mid B]]$, and Step~2 gives the equality.
\end{proof}

\section{Numerical illustration}
\label{sec:numerics}

We illustrate Theorem~\ref{thm:W} in the Rosenblatt case. We take the covariance $R(x)=(1+x^2)^{-\alpha/2}$ of Example~\ref{ex:cauchy} with $\alpha=0.3$, and $\Phi=H_2$, so that $m=2$, $H=0.7$, $\kappa=1$, $c_2=2$ and $\beta^2=2/(H(2H-1))\approx7.143$. For $\eps\in\{10^{-1},10^{-2},10^{-3},10^{-4}\}$, we sample $g$ on $[0,1/\eps]$ on a grid of step $0.1$ by the circulant embedding method, and compute $W_\eps(1)=\eps^{1-\alpha}\int_0^{1/\eps}(g(z)^2-1)\,\dd z$ by the rectangle rule, with $20\,000$ independent samples for each $\eps$. Between $23\%$ and $39\%$ of the eigenvalues of the embedding matrices are negative, of absolute value at most $10^{-3}$ for $\eps=10^{-1}$ and at most $4\cdot10^{-5}$ for $\eps\le10^{-2}$; we set them to zero, which changes the covariance by less than $2\cdot10^{-4}$.

By Lemma~\ref{lem:second_chaos}, $Z(1)=\sum_k\mu_k(\xi_k^2-1)$, where $(\mu_k)$ are the eigenvalues of the operator on $L^2([0,1])$ with kernel proportional to $|u-v|^{H-1}$, normalised so that $2\sum_k\mu_k^2=1$~\cite{VeilletteTaqqu2013}. We compute these eigenvalues by a Galerkin discretisation with $4000$ piecewise constant functions, add a Gaussian term with the variance of the omitted eigenvalues, and obtain the density of $Z(1)$ by numerical inversion of its characteristic function $\theta\mapsto\prod_k(1-2i\theta\mu_k)^{-1/2}e^{-i\theta\mu_k}$.

Figure~\ref{fig:rosenblatt} compares the histograms of $W_\eps(1)/\beta$ with this density. The ratio $\Var(W_\eps(1))/\beta^2$, computed exactly from the covariance, equals $0.630$, $0.848$, $0.939$ and $0.976$ for the four values of $\eps$; the relative error decreases like $\eps^{1-2\alpha}=\eps^{0.4}$, because $R^2-|x|^{-2\alpha}$ is integrable and contributes a term of relative order $\eps^{1-m\alpha}$ to the variance. The Kolmogorov--Smirnov distance between the empirical distribution of $W_\eps(1)/\beta$ and the distribution of $Z(1)$ is $0.218$, $0.080$, $0.033$ and $0.014$. For $\eps=10^{-4}$, the empirical skewness and excess kurtosis of $W_\eps(1)$ are $2.13$ and $7.77$, with bootstrap standard errors of about $0.07$ and $0.7$. They are compatible with the values $2.07$ and $7.63$ of the Rosenblatt variable, given by $2\sqrt2\sum_k\mu_k^3/(\sum_k\mu_k^2)^{3/2}$ and $12\sum_k\mu_k^4/(\sum_k\mu_k^2)^2$ and computed from the Galerkin eigenvalues, extrapolated in the size of the discretisation. For $\eps\le10^{-3}$, the histograms are close to the density, including the heavy right tail; the left tail of $Z(1)$ is very light. Indeed, by~\eqref{eq:q_chi} with $y=1$ and the density argument of Lemma~\ref{lem:top_eigenvalue}, the eigenvalues $\mu_k$ are non-negative, and Lemma~\ref{lem:cond_exp}, applied with the coefficients $-\lambda\mu_k\le0$, gives $\E[e^{-\lambda Z(1)}]\le e^{\lambda^2\sum_k\mu_k^2}=e^{\lambda^2/2}$ for $\lambda\ge0$, hence $\PP(Z(1)<-x)\le e^{-x^2/2}$ for $x\ge0$. For $\eps=10^{-1}$, the left tail is truncated, because $W_\eps(1)\ge-\eps^{-\alpha}$.

\begin{figure}[htbp]
\centering
\includegraphics[width=\linewidth]{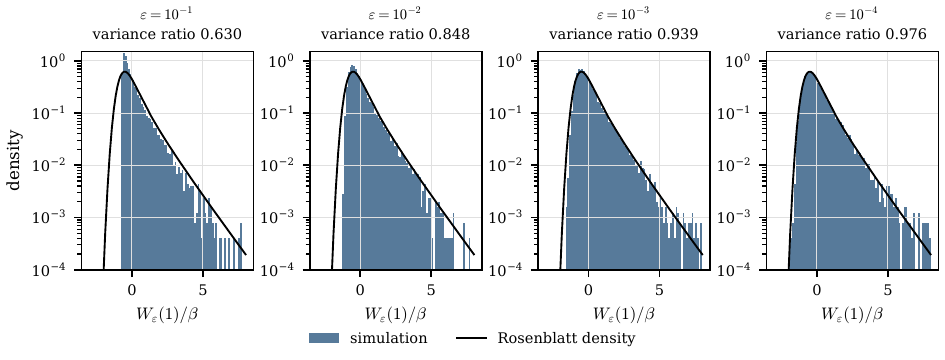}
\caption{Histograms (logarithmic scale) of $W_\eps(1)/\beta$ for $\Phi=H_2$, $R(x)=(1+x^2)^{-0.15}$ and $\eps=10^{-1},\dots,10^{-4}$ (left to right), $20\,000$ samples each, compared with the density of the standardised Rosenblatt variable $Z(1)$ with $H=0.7$ (solid line). The variance ratio is $\Var(W_\eps(1))/\beta^2$.}
\label{fig:rosenblatt}
\end{figure}

\section{Concluding remarks}
\label{sec:conclusion}

The identity~\eqref{eq:Ito_identity} reduces the homogenisation of~\eqref{eq:pde} to a limit theorem for the integrated potential in a weighted space of continuous functions, and it requires neither local times of fractional processes nor exponential moments with respect to the potential. We list some directions that this approach leaves open.

\begin{enumerate}[label=(\arabic*)]
\item \emph{Higher dimensions.} The identity~\eqref{eq:Ito_identity} is one-dimensional. In dimension $d\ge2$, the potential would have to be written as the divergence of a vector field, and the limit is expected to require a renormalisation, as in~\cite{GuBal2016} for $m=1$.
\item \emph{The critical case $m\alpha=1$.} The integrated potential then has a logarithmic normalisation and a Gaussian limit; the identity~\eqref{eq:Ito_identity} still applies, and the question is the tightness in a weighted space.
\item \emph{Moments of the limit for $m\ge3$.} For $m\le2$, Theorem~\ref{thm:mean} and Proposition~\ref{prop:moments} decide whether the moments of $u(t,x)$ are finite. For $m\ge3$, the conditional law of the exponent given $B$ is in the $m$-th Wiener chaos, whose tails are heavier than exponential, and we expect $\E[u(t,x)]=+\infty$, at least for odd $m$ and whatever the sign of $c_m$.
\item \emph{Rates of convergence.} A quantitative version of Theorem~\ref{thm:u}, in a distance between laws on $\mathcal C$, would require a rate in Theorem~\ref{thm:W} in the weighted space, together with a quantitative version of Lemma~\ref{lem:U_continuous}.
\end{enumerate}

\section*{Acknowledgements}

The author thanks the referees of earlier versions of this work, whose comments led to a complete revision of the approach. This research was conducted at the Toulouse School of Economics, Universit\'e Toulouse Capitole. The author acknowledges the support of the French National Research Agency (ANR) under grant ANR-17-EURE-0010 (Investissements d'Avenir program).


\begin{thebibliography}{99}

\bibitem{Bal2009}
G.~Bal.
\newblock Convergence to {SPDE}s in {S}tratonovich form.
\newblock \emph{Comm. Math. Phys.}, 292(2):457--477, 2009.

\bibitem{Bal2010}
G.~Bal.
\newblock Homogenization with large spatial random potential.
\newblock \emph{Multiscale Model. Simul.}, 8(4):1484--1510, 2010.

\bibitem{Bal2011}
G.~Bal.
\newblock Convergence to homogenized or stochastic partial differential equations.
\newblock \emph{Appl. Math. Res. Express. AMRX}, 2011(2):215--241, 2011.

\bibitem{BalGu2015}
G.~Bal and Y.~Gu.
\newblock Limiting models for equations with large random potential: a review.
\newblock \emph{Commun. Math. Sci.}, 13(3):729--748, 2015.

\bibitem{DobrushinMajor1979}
R.~L. Dobrushin and P.~Major.
\newblock Non-central limit theorems for non-linear functionals of {G}aussian fields.
\newblock \emph{Z. Wahrsch. Verw. Gebiete}, 50(1):27--52, 1979.

\bibitem{DombryGuillotin2009}
C.~Dombry and N.~Guillotin-Plantard.
\newblock A functional approach for random walks in random sceneries.
\newblock \emph{Electron. J. Probab.}, 14:no.~50, 1495--1512, 2009.

\bibitem{FrizVictoir2010}
P.~K. Friz and N.~B. Victoir.
\newblock \emph{Multidimensional Stochastic Processes as Rough Paths}.
\newblock Cambridge Studies in Advanced Mathematics, vol.~120. Cambridge University Press, 2010.

\bibitem{GRR1970}
A.~M. Garsia, E.~Rodemich, and H.~Rumsey, Jr.
\newblock A real variable lemma and the continuity of paths of some {G}aussian processes.
\newblock \emph{Indiana Univ. Math. J.}, 20:565--578, 1970/71.

\bibitem{GuBal2012}
Y.~Gu and G.~Bal.
\newblock Random homogenization and convergence to integrals with respect to the {R}osenblatt process.
\newblock \emph{J. Differential Equations}, 253(4):1069--1087, 2012.

\bibitem{GuBal2016}
Y.~Gu and G.~Bal.
\newblock Weak convergence approach for parabolic equations with large, highly oscillatory, random potential.
\newblock \emph{Ann. Inst. Henri Poincar\'e Probab. Stat.}, 52(1):261--285, 2016.

\bibitem{HuHuangNualartTindel2015}
Y.~Hu, J.~Huang, D.~Nualart, and S.~Tindel.
\newblock Stochastic heat equations with general multiplicative {G}aussian noises: {H}\"older continuity and intermittency.
\newblock \emph{Electron. J. Probab.}, 20:no.~55, 1--50, 2015.

\bibitem{HuNualartSong2011}
Y.~Hu, D.~Nualart, and J.~Song.
\newblock Feynman--{K}ac formula for heat equation driven by fractional white noise.
\newblock \emph{Ann. Probab.}, 39(1):291--326, 2011.

\bibitem{IftimiePardouxPiatnitski2008}
B.~Iftimie, \'E.~Pardoux, and A.~Piatnitski.
\newblock Homogenization of a singular random one-dimensional {PDE}.
\newblock \emph{Ann. Inst. Henri Poincar\'e Probab. Stat.}, 44(3):519--543, 2008.

\bibitem{KomorowskiNieznaj2010}
T.~Komorowski and E.~Nieznaj.
\newblock On the asymptotic behavior of solutions of the heat equation with a random, long-range correlated potential.
\newblock \emph{Potential Anal.}, 33(2):175--197, 2010.

\bibitem{LangXanh1983}
R.~Lang and N.~X. Xanh.
\newblock Strongly correlated random fields as observed by a random walker.
\newblock \emph{Z. Wahrsch. Verw. Gebiete}, 64(3):327--340, 1983.

\bibitem{LechihebHaouala2018}
A.~Lechiheb and E.~Haouala.
\newblock Fluctuations in 1{D} stochastic homogenization of pseudo-elliptic equations with long-range dependent potentials.
\newblock \emph{Random Oper. Stoch. Equ.}, 26(1):23--37, 2018.

\bibitem{LechihebNourdinZhengHaouala2018}
A.~Lechiheb, I.~Nourdin, G.~Zheng, and E.~Haouala.
\newblock Convergence of random oscillatory integrals in the presence of long-range dependence and application to homogenization.
\newblock \emph{Probab. Math. Statist.}, 38(2):271--286, 2018.

\bibitem{NourdinPeccati2012}
I.~Nourdin and G.~Peccati.
\newblock \emph{Normal Approximations with {M}alliavin Calculus: From {S}tein's Method to Universality}.
\newblock Cambridge Tracts in Mathematics, vol.~192. Cambridge University Press, 2012.

\bibitem{Nualart2006}
D.~Nualart.
\newblock \emph{The {M}alliavin Calculus and Related Topics}.
\newblock Second edition. Springer, Berlin, 2006.

\bibitem{PardouxPiatnitski2012}
\'E.~Pardoux and A.~Piatnitski.
\newblock Homogenization of a singular random one-dimensional {PDE} with time-varying coefficients.
\newblock \emph{Ann. Probab.}, 40(3):1316--1356, 2012.

\bibitem{PipirasTaqqu2010}
V.~Pipiras and M.~S. Taqqu.
\newblock Regularization and integral representations of {H}ermite processes.
\newblock \emph{Statist. Probab. Lett.}, 80(23-24):2014--2023, 2010.

\bibitem{PipirasTaqqu2017}
V.~Pipiras and M.~S. Taqqu.
\newblock \emph{Long-Range Dependence and Self-Similarity}.
\newblock Cambridge University Press, 2017.

\bibitem{RevuzYor1999}
D.~Revuz and M.~Yor.
\newblock \emph{Continuous Martingales and {B}rownian Motion}.
\newblock Third edition. Springer, Berlin, 1999.

\bibitem{Taqqu1975}
M.~S. Taqqu.
\newblock Weak convergence to fractional {B}rownian motion and to the {R}osenblatt process.
\newblock \emph{Z. Wahrsch. Verw. Gebiete}, 31:287--302, 1975.

\bibitem{Taqqu1979}
M.~S. Taqqu.
\newblock Convergence of integrated processes of arbitrary {H}ermite rank.
\newblock \emph{Z. Wahrsch. Verw. Gebiete}, 50(1):53--83, 1979.

\bibitem{Tudor2008}
C.~A. Tudor.
\newblock Analysis of the {R}osenblatt process.
\newblock \emph{ESAIM Probab. Stat.}, 12:230--257, 2008.

\bibitem{VeilletteTaqqu2013}
M.~S. Veillette and M.~S. Taqqu.
\newblock Properties and numerical evaluation of the {R}osenblatt distribution.
\newblock \emph{Bernoulli}, 19(3):982--1005, 2013.

\bibitem{Young1936}
L.~C. Young.
\newblock An inequality of the {H}\"older type, connected with {S}tieltjes integration.
\newblock \emph{Acta Math.}, 67(1):251--282, 1936.

\end{thebibliography}
\end{document}